\documentclass[10pt]{article}

\usepackage{amsmath}
\usepackage{amsthm}
\usepackage{url}
\usepackage{amssymb}
\usepackage{amsfonts}
\usepackage{authblk}
\usepackage{graphicx}
\usepackage{verbatim}
\usepackage{mathrsfs}
\usepackage{subfigure}
\usepackage{hyperref}
\usepackage{fancyhdr}
\usepackage{latexsym}
\usepackage{graphicx}
\usepackage{dsfont}
\usepackage{color}
\usepackage{fancybox}
\usepackage{framed, color}

\definecolor{shadecolor}{rgb}{1, 0.8, 0.3}

\theoremstyle{plain}
\newtheorem{theorem}{Theorem}[section]

\newtheorem{lemma}[theorem]{Lemma}

\newtheorem{remark}[theorem]{Remark}

\theoremstyle{definition}
\newtheorem{definition}[theorem]{Definition}

\newcommand{\bydef}{\,\stackrel{\mbox{\tiny\textnormal{\raisebox{0ex}[0ex][0ex]{def}}}}{=}\,}

\title{Fourier-Taylor Parameterization of Unstable Manifolds
for Parabolic Partial Differential Equations: Formalism, Implementation 
and Rigorous Validation}

\author[1]{Christian Reinhardt  \thanks{Email: \tt c.p.reinhardt@vu.nl, partially supported by NWO} }

\author[2]{J.D. Mireles James \thanks{J.M.J partially supported by NSF grant DMS - 1318172
Email: {\tt jmirelesjames@fau.edu}}}

\affil[1]{Vrije Universiteit Amsterdam, Department of Mathematics}
\affil[2]{Florida Atlantic University, Department of Mathematical Sciences}

\date{\today}

\begin{document}
\maketitle

\begin{abstract}
In this paper we study high order expansions of 
chart maps for local finite dimensional unstable manifolds of hyperbolic equilibrium solutions of 
scalar parabolic partial differential equations.  Our approach is based on studying an 
infinitesimal invariance equation
for the chart map that recovers the
dynamics on the manifold in terms of a simple conjugacy. We develop formal series
solutions for the invariance equation and efficient numerical methods for 
computing the series coefficients to any desired finite order.
We show, under
mild non-resonance conditions, that the formal series expansion   converges in a 
small enough neighborhood of the equilibrium. 
An a-posteriori computer assisted argument proves
convergence in larger neighborhoods.  
We implement the method for a spatially inhomogeneous Fisher's equation
and numerically compute and validate 
high order expansions of some local unstable manifolds for morse 
index one and two.  
We also provide
a computer assisted existence proof of a saddle-to-sink 
heteroclinic connecting orbit.    
\end{abstract}
\vspace{20pt}
\newpage

\tableofcontents

\section{Introduction} \label{sec:intro}
Global analysis of nonlinear parabolic PDEs, from the dynamical 
systems point of view, begins by studying phase space landmarks such as
stationary and periodic solutions.
Once the existence, stability, and analytic properties of these
are catalogued, one wants to understand how the 
landmarks fit together and organize the phase space.
Classical dynamical systems theory for parabolic PDEs 
tells us that the phase space is organized by global invariant
objects such as heteroclinic connecting orbits and inertial 
manifolds, and a necessary first step toward understanding these
is to study the unstable manifolds of the landmarks. 
These unstable manifolds are necessarily 
finite dimensional, as the semi-flow generated by a parabolic
PDE is compact.

The present work deals with the numerical approximation of   
unstable manifolds of equilibrium solutions of scalar parabolic PDE.
Our approach is based on the parameterization method 
of \cite{param1, param2, param3}, which provides a general
functional analytic framework for studying \textit{non-resonant}
invariant manifolds in Banach spaces.  We refer also to the 
overview in \cite{jpJordiMarcioRafaPaper}, where the Parameterization 
Method for parabolic PDE is discussed in great generality (indeed 
this reference suggests the approach of the present work).   
The idea is to formulate a functional equation 
whose solutions are chart maps for the unstable manifold.
As suggested in \cite{jpJordiMarcioRafaPaper}, we exploit an infinitesimal 
conjugacy equation which depends explicitly on the form of the PDE
but does not involve the flow.  Because the parameterization 
satisfies a conjugacy, our method recovers the dynamics on the 
manifold in addition to the embedding.  
We develop a formal series solution of the conjugacy equation,
and implement a numerical scheme for computing the coefficients of 
the series to any desired order.

High order approximations are useful for studying the unstable manifold far 
from its equilibrium. Yet numerically evaluating a high order expansion 
far from the equilibrium raises concerns about accuracy.
The main result of the present work is a computer assisted argument which 
provides mathematically rigorous error bounds for high order approximations.
The argument does not require restricting the approximation to a 
small neighborhood of the equilibrium. 
Rather, we develop a-posteriori tools which 
use in a fundamental way that the numerical 
representation of the manifold approximately solves a functional equation.

The problem is infinite dimensional, and 
in order for our argument to succeed it is critical that we manage 
a number of errors introduced by the finite dimensional truncations.
 In the present work this truncation error analysis is 
 facilitated by two observations.
 First, the compactness/smoothing properties of 
the parabolic PDE allow us to control the spatial/spectral truncation.
Indeed  the computer assisted proofs implemented in 
Section \ref{sec:appl} make
substantial use of the fact that the PDE is formulated on a geometrically
simple domain, where the eigenexpansion of the differential operator is given explicitly in 
terms of Fourier (cosine) series. 
Second, the Parameterization Method admits certain free parameters 
(namely the scalings of the unstable eigenvectors) in the formulation of conjugacy 
equation, and these scalings control the decay rate of 
the formal series coefficients.  We exploit this control over the decay  
to insure that the truncated series expansion satisfies
some prescribed error tolerance. The second consideration is fundamental 
 to the Parameterization Method, and has nothing to do with the particular
 eigenbasis for the PDE or even the fact that we consider parabolic problems.

\begin{remark}[Computer assisted proof for equilibria of PDEs] \label{rem:CAP_PDE}
{\em

Establishing existence and stability of stationary solutions to PDEs is a subtle business. 
When the nonlinearities are 
strong and the PDE is far from a perturbative regime, it may be impossible to carry 
out this analysis analytically.  Numerical simulations provide valuable insight into the dynamics 
of PDEs, and in recent years substantial effort has gone into developing computer 
assisted methods of proof which validate simulation results.  

A thorough review of the literature on computer assisted proof for of PDEs
would lead us far afield of the present discussion.  We refer to the works of 
\cite{MR1838755, MR3390404, yamamoto,
MR2679365, MR2852213, MR3338319,
MR2589483, MR1810529, MR3203775,
SJPKons, smoothbranchesJBJPKons, JPbredenvanicat, AlJPJay}
for fuller discussion of computer assisted proof for 
equilibrium solutions of PDEs, and  
also \cite{MR3196951, matsuePDE, MR2874050, MR2852213, robertoEigPaper}
for more discussion of techniques for validated computation 
of eigenvalue/eigenvector pairs for infinite dimensional problems.
Let us also mention the review articles of 
\cite{MR3444942, jayKonsReview,MR1420838} and the book 
of \cite{MR2807595} for broader overview of the field.  
While the list of references given above is far from exhaustive
(in particular the list ignores the growing literature on computer
assisted proof for periodic orbits of PDEs), 
it is our hope that these works and the references discussed therein
could help the interested reader wade into the literature. 
}
\end{remark}

\begin{remark}[Computer assisted proof for unstable manifolds in finite dimensions] 
\label{rem:oldWork}{\em
It must also be noted that the present work builds on a growing body of 
literature devoted to validated numerical methods for studying stable/unstable 
manifolds of equilibrium solutions for finite dimensional vector fields.  
A thorough review is beyond the scope of the present work, and we direct the 
reader to  
\cite{MR3281845, MR3022075, MR2494688, cosy1, JBJPJayKons, MR2773294, 
MR3443692, MR3032848, Jayparameters, parmResPaper, MR3437754, MR2902618}
for more complete discussion of the literature.  This list
ignores works devoted to validated numerical methods for stable/unstable 
manifolds of discrete time dynamical systems and also validated methods for computing other 
types of invariant manifolds (for example invariant tori and their stable/unstable manifolds).
Again, we refer to the review articles mentioned in 
Remark \ref{rem:CAP_PDE}.
}
\end{remark}

\subsection{A family of examples}
In order to minimize the proliferation
of notational difficulties, we consider a fixed specific class of 
scalar parabolic equations.

More precisely, assume the PDE is of the form 
\begin{equation}\label{eq:gen_PDE}
u_t = Au + \displaystyle\sum_{n = 1}^{s}c_n(x) u^n ,\quad u = u(x,t)\in\mathbb{R}, \quad(x,t)\in I\times \mathbb{R}_+
\end{equation}
where $I\subset \mathbb{R}$ is a compact interval, $A$ is a parabolic differential operator, $s$ is the order of the nonlinearity and $c_n(x)$ are the smooth coefficient functions possibly depending on the spatial variable $x$. Using an orthonormal basis corresponding to the eigenfunctions of $A$ for the particular domain and boundary conditions we translate \eqref{eq:gen_PDE} into a countable system of ODEs.\\

The resulting system of ODEs, projected onto the eigenbasis, is of the form  
\begin{equation}\label{eq:gen_ODEsys}
a_k'(t) = \mu_k a_k + \displaystyle\sum_{n = 1}^{s} \displaystyle\sum_{\stackrel{\sum k_i = k}{k_i\in\mathbb{Z}}}(c_n)_{|k_1|}a_{|k_2|}\cdots a_{|k_{n+1}|}\bydef g_k(a) \quad k\geq 0
\end{equation}
where $\mu_k$ are the eigenvalues of $L$  and $a = (a_k)_{k\geq 0}$ are the expansion coefficients of $u$  in the respective eigenbasis. We use the shorthand notation $a' = g(a)$ for \eqref{eq:gen_ODEsys}.  To define the unstable manifold we are interested in, assume $\tilde{a}$ to be given such that $g(\tilde{a}) = 0$. Its unstable manifold is given by
\begin{equation}
W^{u}(\tilde{a}) = \{a_0 : \exists \text{ solution $a(t)$ of \eqref{eq:gen_ODEsys} } :
a(0) = a_{0}\quad \displaystyle\lim_{t\to-\infty}a(t) = \tilde{a}\}.
\end{equation}
It is a classical fact that for scalar parabolic
PDEs of the form \eqref{eq:gen_PDE}, $W^{u}(\tilde a)$ is a finite dimensional manifold \cite{sellyu}.

As a concrete application consider  
the boundary value problem for the following reaction diffusion equation on a 
one-dimensional bounded spatial domain with Neumann boundary conditions: 
\begin{equation}\label{eq:Fishers}
\begin{aligned}
u_{t} &= u_{xx} + \alpha u(1- c_2(x) u), \quad (x,t)\in[0,2\pi]\times \mathbb{R},\\
&u_{x}(0,t)= u_{x}(2\pi,t) = 0\quad \forall t\geq 0 
\end{aligned}
\end{equation}
Here $\alpha>0$ is a real parameter and $c_2(x)>0$ is a spatial inhomogeneity.  We consider both the case $c_2(x) = 1$ and $c_{2}(x)$ non-constant, specifically a Poission kernel. For notational convenience we drop the index $2$ and refer to the spatial inhomogeneity as $c(x)$. Moreover the parameter $\alpha$ has the role of an eigenvalue parameter to consider different dimension configuration of  the unstable manifolds at hand.   The equation is known as Fisher's equation, or as the Kolmogorov-Petrovsky-Piscounov equation,
and has applications in mathematical ecology, genetics, and the theory of 
Brownian motion \cite{MR1423804, MR639998, MR0400428}.

\subsection{Methodology of the present work: sketch of the approach}\label{sec:methods}
Let $\tilde a$ be an equilibrium solution of  \eqref{eq:gen_ODEsys} 
with known Morse index and eigendata. More 
precisely suppose that $Dg(\tilde a)$ has exactly
$d$ unstable eigenvalues $\tilde{\lambda}_j$.
In the present work we assume that the unstable eigenvalues
are real, and that each has multiplicity one.  Then let $\tilde{\xi}_{j}$, $1 \leq j \leq d$
denote an associated choice of unstable eigenvectors,
i.e. assume that 
\begin{equation}\label{eq:eigs_seq}
Dg(\tilde{a})\tilde{\xi}_j = \tilde{\lambda}_j \tilde{\xi}_j \quad j = 1,\ldots,d.
\end{equation} 
In practice the first order data is not explicitly given, 
and we perform a sequence of preliminary computer
assisted proofs in order to verify that the assumptions 
are satisfied.  We refer the reader again to the references mentioned in 
Remark \ref{rem:CAP_PDE} above, and also to Sections 
\ref{sec:nonlinearAnalysis} and \ref{sec:appl} of the present work for more 
refined discussion of these preliminary considerations.

We are now ready to give an informal description of the Parameterization Method for 
unstable manifolds.  See also \cite{jpJordiMarcioRafaPaper}.
Let 
\[
\mathbb{B}_1 := \{(\theta_1, \ldots, \theta_d) \in \mathbb{R}^d \colon
|\theta_j| < 1, 1 \leq j \leq d \}.
\]
We seek solutions of the functional equation 
\begin{equation}\label{eq:func_eq}
g(P(\theta_1, \ldots, \theta_d)) =
\tilde \lambda_1 \theta_1 \frac{\partial}{\partial \theta_1} P(\theta_1, \ldots, \theta_d)
+ \ldots +
\tilde \lambda_d \theta_d 
\frac{\partial}{\partial \theta_d} P(\theta_1, \ldots, \theta_d),
\end{equation}
for all $\theta = (\theta_1, \ldots, \theta_d) \in \mathbb{B}_1$
satisfying the linear constraints
\begin{subequations}\label{eq:linear_constraints}
\begin{alignat}{1}
P(0) &= \tilde{a}\\
 \frac{\partial}{\partial \theta_j}P(0) &= \tilde{\xi}_j,
 \quad \mbox{for } 1 \leq j \leq d.
\end{alignat}
\end{subequations}
Note that Equation \eqref{eq:func_eq} is actually a 
Banach spaced valued partial differential equation (or a system of infinitely many
scalar partial differential equations when the Banach space is a sequence space).
We refer to Equation \eqref{eq:func_eq} as \textit{the invariance equation},
and note that it is expressed more concisely as
\begin{equation} \label{eq:func_eq_highLevel}
(g \circ P)(\theta) = DP(\theta) A_u \theta.
\end{equation}
Here $A_u$ is the $d \times d$ diagonal matrix of unstable eigenvalues.  
Equation \eqref{eq:func_eq_highLevel} makes it clear that 
the vector field  $g$ is tangent to the image of $P$, i.e. $P$ parameterizes an invariant 
manifold.  
Indeed we have the following lemma, which makes precise the claim that 
$P$ recovers the dynamics on the manifold.
\begin{lemma} \label{lem:flowInv}
Assume that $P$ solves \eqref{eq:func_eq} and
satisfies the first order constraints of \eqref{eq:linear_constraints}. 
Then for every $\theta\in\mathbb{B}_1$ the function 
\begin{equation}
a(t) = P(\exp(A_ut)\theta)
\end{equation}
solves $a' = g(a)$ for all $t\in(-\infty,T(\theta))$ for a positive time 
$T(\theta)$. In particular $\displaystyle\lim_{t\to-\infty}a(t) = \tilde{a}$.
\end{lemma}

\noindent The proof is obtained by direct computation using that $\text{real}(\lambda_i)>0$ for $i = 1,\ldots,d$
(see Lemma 2.1 in \cite{parmSlowVectBundles} and Lemma 2.6 in \cite{parmResPaper} for 
elementary proofs in finite dimensional contexts).

In order to obtain an approximate solution of Equation \eqref{eq:func_eq}
we adopt the power series ansatz
\begin{equation}\label{eq:pow_ans}
P(\theta) = \sum_{|m| = 0}^{\infty} p_m \theta^m.
\end{equation} 
Here $m\in\mathbb{N}^d$ is a multi-index, $\theta^m := \theta_1^{m_1}\cdots\theta_d^{m_d}$,
and $|m| = m_1 + \ldots + m_d$.
Plugging \eqref{eq:pow_ans} into \eqref{eq:func_eq}
and matching like powers of $\theta$ leads to a system of 
infinitely many coupled nonlinear equations for the Taylor coefficients
$\{p_m\}_{m \in \mathbb{N}^d}$.  The details are given in Section
\ref{sec:homological}, in particular see Equation \eqref{eq:homologicalEqnBanachAlg}.

Truncating leads to a system of finitely many coupled nonlinear
scalar  equations which are solved (for example) by a numerical Newton method,
leading to  approximate
Taylor coefficients $\{\bar p_m\}_{|m| = 0}^M$ where $p_m \in \mathbb{R}^K$ 
for each $0 \leq |m| \leq M$.    
The numerical procedure is discussed in Section \ref{sec:zero_fin}, with 
application in Section \ref{sec:appl}.  After this computation we have a  
finite dimensional approximate parametrization of the form
\begin{equation}\label{eq:fin_map}
P^{MK}(\theta) = \displaystyle\sum_{m\in\mathcal{F}_M}\bar{p}_m\theta^m.
\end{equation}

\begin{remark}[Rescalings] \label{rem:rescalings}{\em
The choice of domain $\mathbb{B}_1$ deserves some explanation.  
We will see in Section \ref{sec:ParmMethod} that solutions of Equation
\eqref{eq:func_eq} are unique up to the choice of the eigenvectors
$\xi_1, \ldots, \xi_j$, so that rescaling the eigenvectors leads to 
parameterizations of larger or smaller local portions of the unstable manifold.
On the other hand, one can imagine controlling the size of the 
local portion by fixing the scalings of the eigenvectors 
(say with unit norm) and varying instead the size of the domain of $P$.  
The two approaches are dynamically equivalent.  Nevertheless rescaling the 
eigenvectors allows us to control also the decay rate of the Taylor coefficients
of $P$, and this stabilizes the problem numerically.  
Hence we fix once and for all the domain $\mathbb{B}_1$,
and ask in a particular problem ``what is the best choice of the scalings
for the eigenvectors?''  The answer depends on the problem at hand
and is only answered after some numerical calculations.  We return to 
this question in Section \ref{sec:appl}}.
\end{remark}

We now come to the question: how good is this approximation?
Define the defect or \textit{a-posteriori error} for the problem by
\[
\epsilon_{MN} := \sup_{\theta \in \mathbb{B}_1} 
\| g[P^{MN}(\theta)] - D P^{MN}(\theta) A_u \theta\|,
\]
in an appropriate norm to be specified later (in practice we also
 ``pre-condition'' the defect with a smoothing approximate inverse.  See 
Section \ref{sec:nonlinearAnalysis}). 
Our task is to establish sufficient conditions, depending on 
$g$, $\tilde a$, $\tilde{\lambda}_1, \ldots, \tilde{\lambda}_d$, $P^{MN}$, and the underlying  
 Banach space, so that 
$\epsilon_{MN} \ll 1$ implies the existence of a true solution 
$P$ of Equation \eqref{eq:func_eq} on $\mathbb{B}_1$. 
In fact our argument will show that
\begin{equation}\label{eq:uniform_bound}
\displaystyle\sup_{\theta\in\mathbb{B}_1}\|P(\theta)-P^{MK}(\theta)\|\leq r_{P}, 
\end{equation}
where $P$ is the true solution of Equation \eqref{eq:func_eq}, and 
the explicit value of $r_P$ (which depends on the defect) 
comes out of our argument.  It is critical that we formulate
sufficient conditions which can be checked via 
carefully managing floating point computations.
Floating point checks of the hypotheses employ \textit{interval arithmetic}
in order to guarantee that we obtain mathematically rigorous results 
\cite{MR0231516, RumpIntlab, MR2807595}.
 
In order to formulate the sufficient conditions just discussed
we implement a computer assisted a-posteriori scheme
which has its roots in the the seminal work of 
\cite{MR648529, MR727816} on the Feigenbaum conjectures.
We study the equation 
\[
f(P(\theta)) = (g \circ P)(\theta) - DP(\theta) A_u \theta = 0,
\]
via a modified Newton-Kantorovich argument. 
More precisely, we show that the ``Newton-like'' operator 
\[
T(P) := P - A f(P),
\]
is a contraction on a  ball of radius $r_P$ in the Banach space
about the approximation solution $P^{MN}$. Here $A$ is a problem dependent 
approximate inverse of $Df(P^{MN})$, which we choose based on both numerical and analytic considerations.
The argument is formalized in Section 
\ref{sec:nonlinearAnalysis},  with implementation discussed in Section \ref{sec:appl}.

\subsection{Discussion}\label{sec:discussion}

In practice we learn that the argument outlined in Section \ref{sec:methods}
succeeds or fails only after attempting the a-posteriori validation, 
and these attempts are  computationally expensive.  
With this in mind we include the a-priori Theorem \ref{thm:smt}
in Section \ref{sec:aPrioriExist}.
The theorem says that, given some mild non-resonance
conditions between the unstable eigenvalues (assumptions which are made precise in 
Definition \ref{def:nonRes}) there exist choices of eigenvectors so that 
Equation \eqref{eq:func_eq} has a solution.  The solution is unique up to the 
choice of the eigenvectors, but a-priori may parameterize only a small portion
of the local unstable manifold near the equilibrium.  Our Theorem is an infinite 
dimensional generalization of the results in Section $10$ of \cite{param3}.

The proof of the Theorem \ref{thm:smt} assures existence only
for a small enough choice of the scalings of the unstable eigenvectors.
On the other hand, in Section \ref{def:nonRes} we see that the 
size of the local manifold parameterized by $P$ is determined 
in a rather explicit way by the size of these scalings.  
In applications we would like 
to choose the eigenvector scalings as large as possible, so that we 
learn more about the unstable manifold far from the equilibrium 
solution.  This desire must be weighed against the fact that 
for larger choices of the scalings we 
risk loosing control of the convergence of the series.

Viewed in this light, the tools of the present work provide a 
mathematically rigorous computer assisted method for pushing
the existence results as far as possible in specific applications.
Since the desired parameterization exists in a small enough neighborhood of 
the equilibrium solution by Theorem \ref{thm:smt}, our argument has a ``continuation'' flavor  
(the continuation parameters being the eigenvector scalings, which 
in turn govern the size of the local unstable manifold in phase space).
The novelty is that we care only about rigorous results at the end of
the continuation.  We argue in 
Sections \ref{sec:ParmMethod} and \ref{sec:appl} that expensive
validation computations can be postponed until we are all but certain 
the computer aided proof will succeed.  
These considerations are also discussed at length for finite dimensional
vector fields in \cite{MR3437754}.

\begin{remark}[Extensions of the a-priori results]\label{rem:extnesionsIntro}
{\em
A word about the technical assumptions of Theorem \ref{thm:smt}.  
In addition to the non-resonance conditions 
postulated in Definition \ref{def:nonRes},
Theorem \ref{thm:smt} postulates that each of the 
unstable eigenvalues is real and that each has multiplicity one.
Moreover we assume that the nonlinearity is given by an analytic function.  

We remark that these assumptions are technically convenient, 
and should not be interpreted as fundamental restrictions.  For example 
real invariant manifolds associated with complex conjugate eigenvalues 
are treated exactly as discussed in \cite{MR3207723}.
Indeed it is possible to remove completely the 
non-resonance and multiplicity conditions, as long as there are ``spectral gaps''
(eigenvalues bounded away from the imaginary axis). 
The necessary modification is to conjugate the parameterization to a polynomial,
rather than a linear vector field. The reader interested in this technical extension
could consult the work of \cite{parmResPaper} for the case of finite dimensional 
vector fields, and the work of \cite{param1} for the case of infinite dimensional
maps.

One can also ease the regularity requirements,
and assume only that the nonlinearities
are only $C^k$ rather than analytic as in \cite{param1, param2}.
This however changes substantially the flavor of the 
validation scheme developed here, which exploits analyticity
in a fundamental way.
More precisely, the analyticity of the nonlinearity 
allows us to look for analytic parameterizations, which in turn 
allows us to study the Taylor coefficients in  a Banach space of 
rapidly decaying infinite sequences.
If instead we expand the manifold as a finite Taylor polynomial plus a unknown 
remainder function, then the a-posteriori analysis for the remainder 
function must be carried out in function space.

Let us also mention that, following the work of \cite{param1, param2, param3},
it might be possible to extend the methods of the present work 
to problems with continuous spectrum such, as PDEs on unbounded domains.
The extensions remarked upon above are not considered further in the present work.
}
\end{remark}

\begin{remark}[Extension of the formal series results: non-polynomial nonlinearities
and systems of scalar parabolic PDE] {\em
While Theorem \ref{thm:smt} is formulated
for general analytic nonlinearities, the 
formal solution of Equation \eqref{eq:func_eq} developed in Section
\ref{sec:homological} is derived under the further assumption of  
polynomial nonlinearity.  This is not as restrictive as it might seem upon first
glance, as transcendental nonlinearities given by elementary functions
can be treated using methods of automatic differentiation.

A thorough discussion of automatic differentiation as a tool for 
semi-numerical computations and computer
assisted proof is beyond the scope of the present work, and we 
refer the interested reader to the books 
\cite{MR3077153, MR2807595} and also to the work of 
\cite{alexADmanuscript, autoDiffFourier} as an entry point
to the literature.  In terms of the present discussion the relevant point is that 
automatic differentiation allows us to develop formal series evaluation of 
non-polynomial nonlinearities by appending additional differential equations.
When the nonlinearity is among the so called ``elementary functions
of mathematical physics'' the appended equation is polynomial,
and we end up with a system of scalar parabolic PDEs with polynomial nonlinearities.

Extending the methods of the present work to such systems will make
an interesting topic for a future study.  Such a study might
treat automatic differentiation and non-polynomial nonlinearities 
as an application, but could also discuss unstable manifolds for 
systems of reaction diffusion equations
in general.  These topics are not pursued further in the present work.  
}
\end{remark}

\begin{remark}[Spectral versus finite element bases]
{\em 
A more fundamental limitation of the present work is that, 
when it comes to the implementation details in Section \ref{sec:appl}, we restrict 
our attention to the case of spectral bases 
for the spatial dimension of the PDE.  
This allows us to exploit a sequence space analysis 
which is very close to the underlying numerical methods,
and which for example avoids the use of Sobolev inequalities and interpolation 
estimates.  However such sequence space implementation is only 
possible on simple domains where the eigenfunction expansion of the 
linear part of the PDE is explicitly known.

An useful and nontrivial extension of the techniques of the present work
would be to implement an a-posteriori argument  for the Parameterization Method for 
unstable manifolds of parabolic PDEs using finite element basis.  
In such a scheme the sequence space calculus exploited in the present work
would be replaced classical Sobolev theory.  
We believe that this extension is both natural and plausible,
and for this reason we frame the a-priori existence results 
and discussion of formal series in Section \ref{sec:ParmMethod} 
in the setting of  a general Banach algebra.   
}
\end{remark}

\begin{remark}[The role of a-priori spatial regularity]
{\em
It is worth noting that, strictly speaking, we do not need to know a-priori the regularity 
of the unstable manifold: rather this is a convenience.   Indeed, if our method succeeds then we obtain 
regularity results a-posteriori.  In practice it is helpful to have 
an  ``educated guess'' concerning the regularity of the 
manifold, as this informs the choice of norm in which to frame the computer assisted proof.
For more nuanced discussion of computer assisted proof in 
sequence spaces associated with functions
in weaker regularity classes we refer to \cite{ck_analytic}.

This suggests another interesting direction of future study, namely to extend the
 methods of the present work to infinite dimensional settings such as state 
dependent delays, where even the a-priori existence of solutions is often in question.
More discussion of the computer as a tool for 
studying breakdown of regularity of invariant objects can be found in 
the works of \cite{MR2672635, MR2220541, MR3082311, MR2967458}.
}
\end{remark}

\begin{remark}[Implementation] \label{rem:implementation}
{\em
We provide full implementation details 
for the example of a spatially inhomogeneous Fisher equation. 
The implementation is discussed in Section \ref{sec:appl}, and 
involves the derivation of a number of problem dependent 
estimates.  These estimates are then used in order to show that the 
Newton like operator is a contraction in a neighborhood of  
our numerical approximation.  

We choose not to suppress the derivation of these bounds 
for two reasons.  The first is that their inclusion
gives the present work a degree of plausible reproducibility.
Indeed the estimates in Section \ref{sec:appl} can be 
viewed (more or less) as pseudo-code for the computer programs which validate
our approximation of the unstable manifold.  The second reason is 
that including a few pages of estimates
makes entirely transparent the role played by the computer in our arguments. 
One sees that (after the initial stage of numerical approximation)
the computer is primarily used to add and multiply long 
lists of floating point numbers.  If the results satisfy certain 
completely explicit inequalities then 
we have our proof.
}
\end{remark}          

\begin{remark}[Computer assisted existence proofs for connecting orbits of parabolic PDEs]
{\em
As already suggested in the introduction, our primary motivation for validated 
computation of local unstable manifolds is our interest in global dynamics, 
for example heteroclinic connecting orbits between equilibrium solutions 
of parabolic PDE.  In order to demonstrate that the methods of the present work 
are of value in this context, we prove in Section \ref{sec:connorbit} the existence
of a saddle-to-sink connecting orbit for a Fisher equation. 
More precisely we establish the existence of an
orbit which connects an equilibrium solution 
of finite non-zero Morse index to a fully stable equilibrium solution.  

Nevertheless, we want to be clear that the computer assisted existence proof
in Section \ref{sec:connorbit} is a ``proof of concept'', and there remains much work 
to be done if one wants to develop general computer assisted analysis 
for transverse connecting orbits for PDEs. In particular, 
the computations in Section \ref{sec:connorbit}
establish connections only for orbits asymptotic to a sink, i.e. 
we compute explicit lower bounds on 
the size of an absorbing neighborhood of the stable equilibrium state, and then 
we simply check that our parameterized local manifold enters this neighborhood.

In general one has to contend with saddle-to-saddle connections, in which case a 
more subtle analysis of the (non-zero co-dimension) stable manifold is needed.
The non-resonance conditions required for the Parameterization 
 Method seem to rule out the study of finite co-dimension manifolds in infinite 
 dimensional problems.  Nevertheless, 
error bounds for the stable manifold could be obtained by 
implementing the geometric methods  of    
\cite{MR2494688, MR2784613, piotrCOdraft, jacekCOpaper}, 
or adapting the functional analytic approach of \cite{infMapConnections}
to the setting of parabolic PDEs. 
We refer to the work of  \cite{infMapConnections} for more 
discussion of computer assisted existence proofs for saddle-to-saddle connections
in infinite dimensions (though  \cite{infMapConnections} treats explicitly
only the case of infinite dimensional maps).

We also remark that in general it is not enough to establish the 
existence of a ``short-connection'' as we do in Section  \ref{sec:connorbit}.
By a short connection we mean a connecting orbit which is described using only 
parameterizations of the local stable and unstable manifolds (this terminology is 
discussed further in \cite{MR3207723}).  Instead, the 
typical situation is that the rigorously validated local unstable and stable manifolds 
do not intersect.  In this case a computer assisted existence proof for a connecting orbit 
requires the solution of a two point boundary value problem 
whose solution is an orbit segment beginning 
on the unstable and ending on the stable manifold.
We refer to the 
works of \cite{MR2174417, MR2681634, MR2271217,
MR2302059, MR3097025, MR2679365, MR3281845, MR2388394,
MR3207723,hexroll, JBJPJayKons}
for more discussion of computer assisted proof of 
saddle-to-saddle connections in finite dimensional problems.

Finally we direct the interested reader to the work of 
\cite{piotrCOdraft, jacekCOpaper}, where another approach to computer assisted 
proof for connecting orbits in parabolic PDE is given.
The computer aided proofs of connecting orbits in the references just cited
are similar in spirit to those developed in the present work, 
with at least one important difference.  
While the authors of the work just cited also follow
trajectories on a local unstable manifold until they enter a trapping region 
of a sink, the orbit is propagated via mathematically rigorous
numerical integration of the PDE.  Then their approach could be used to 
study substantially longer connecting orbits than those obtained 
with our proof of concept in Section  \ref{sec:connorbit}.
More thorough discussion of rigorous integration of PDEs can be found in 
the works of \cite{MR2788972, MR3167726, MR2728184}.

At the same time, we remark that the authors of \cite{piotrCOdraft, jacekCOpaper} 
represent the unstable 
manifold locally using a linear approximation and obtain validated 
error bounds on this approximation via geometric arguments
based on cone conditions.
An interesting avenue of future study might be to combine
the high order methods of the present work
with methods for rigorous numerical integration of parabolic PDEs as developed in 
\cite{piotrCOdraft, jacekCOpaper, MR2788972, MR3167726, MR2728184} in order to
prove the existence of connecting orbits in more challenging applications.

}
\end{remark}
                                                                                                                                                                                                                                                                    
The paper is organized as follows. First in Section \ref{sec:nands} we discuss the Banach spaces we will be working on.  In Section \ref{sec:nonlinearAnalysis} we discuss the method we use to validate the solution to a zero finding problem $f(x) = 0$ together with the analysis of the linear eigendata of $Df(x)$. Specifically in Section \ref{sec:nonlinearAnalysis}
 we discuss the radii polynomial method from \cite{SJPKons}.  In Section \ref{subsec:morse_index} we demonstrate how to make sure the accurate Morse index of $Df(x)$ is obtained given an approximate derivative $A^{\dag}$ whose spectrum is understood completely. In Section \ref{sec:zero_fin}  we describe how to setup a zero finding problem whose solution corresponds to the power series coefficients of a parametrization of the unstable manifold of a hyperbolic fixed point. In Section \ref{sec:appl} we showcase our method in examples. We discuss Fisher's equation  from \eqref{eq:Fishers}. In Section \ref{sec:firstorderdata} we describe how to compute and validate the first order data at an equilibrium in the specific example, including the validation of the eigendata and the Morse index. In Section  \ref{subsec:cmanifold} we compute a one and two dimensional unstable manifold for a non-trivial equilibrium and the origin respectively.  In Section \ref{sec:connorbit}  we discuss the computer-assisted proof of a short connecting orbit from a fixed point of Morse index 1 to Morse index 0.

 All computer programs used to obtain the results in this work are freely available at the 
 papers home page \cite{parmPDEcode}.

\section{Background} \label{sec:background}

\subsection{A-posteriori analysis for nonlinear operators} \label{sec:nonlinearAnalysis}

Let $X$ and $X'$ be Banach spaces and $f \colon X \to X'$
be a smooth map.  Throughout the sequel we are interested in the
zero finding problem
\[
f(x) = 0.
\]
Suppose we have in hand an approximate solution $\bar x$. 
In our context $\bar x$ is usually the result of a numerical computation.
Our goal is to prove that there exists a true solution nearby.

To this end let $A$ be an injective (one-to-one) bounded linear operator having that
\[
A f(x), ADf(x) \in X
\]
for all $x  \in X$.  
Heuristically we think of $A$ as being a smoothing 
approximate inverse for $Df(\tilde x)$.  In other words, 
we ask neither that $f(x)$ is a self map or 
that $Df(x)$ is a bounded linear operator. Rather, we allow that
$f$ and $Df(x)$ may be unbounded operators and that
 $A$ ``smooths'' $f$ and $Df$, bringing the composition back into $X$.

Now we define the Newton-like operator 
\[
T(x) = x - Af(x),
\]
and note that the fixed points of $T$ are in one-to-one correspondence with  
the zeros of $f$.  
We use the Banach Fixed Point Theorem on a ball  of  radius 
$r$ around an approximate solution $\bar x$ to show that
that $T$ has a unique fixed point in said ball. 

Our approach follows that of \cite{yamamoto}, in that 
we consider the radius $r$ as one of our unknowns, and find a suitable range of 
radii such that $T$ is a contracting selfmap on the corresponding balls
(rather than guessing a value for the radius $r$ 
and applying the  Newton-Kantorovich Theorem).
This is referred to by some authors as the radii-polynomial approach, 
and in addition to the work just cited we
refer the interested reader also to 
the work of \cite{SJPKons, smoothbranchesJBJPKons}
and the references discussed therein.

\begin{definition}\label{def:YZp} \textbf{$Y$, $Z$-bounds and the radii polynomial }\\
Recall the fixed point operator $T$ specified in \eqref{eq:T} corresponding to the zero finding
 map \eqref{eq:f}. Assume an approximate zero $\bar{x}$ to be given. Let us define the following 
 bounds $Y\in\mathbb{R}$ and $Z(r)\in\mathbb{R}[r]$:
\begin{enumerate}
\item the $Y$-bound $Y\in\mathbb{R}$ measuring the residual: 
\begin{equation}\label{eq:Y}
\|T(\bar{x})-\bar{x}\|\leq Y
\end{equation}

\item the $r$-dependent $Z$-bound measuring the contraction rate on a ball of (variable) radius $r$:
\begin{equation}\label{eq:Z}
\sup_{u,v\in\mathbb{B}_{1}}\|DT(\bar{x}+ru)rv\|\leq Z(r)
\end{equation}
\end{enumerate}
Define the polynomial 
\begin{equation}\label{eq:rad_pol}
\beta(r) = Y+Z(r)-r.
\end{equation}
\end{definition} 
The benefit of Definition \ref{def:YZp} is  the following Lemma.

\begin{lemma}\label{eq:radpol_lem}
Assume an approximate zero $\bar{x}$ of $f$ defined in \eqref{eq:f} to be given. If $\beta(r)<0$ with $\beta$ defined in \eqref{eq:rad_pol} for a positive radius $r_{\bar{x}}$, then $T$ given by \eqref{eq:T} is a contraction on $\mathbb{B}_{r_{\bar{x}}}(\bar{x})$. Hence there is a unique zero $\tilde{x}$ with $\tilde{x}\in \mathbb{B}_{r_{\bar{x}}}(\bar{x})$.
\end{lemma}

A proof of this lemma has appeared in many places. See for example \cite{SJPKons}. 
The decisive feature of the condition $\beta(r)<0$ in our context is that after deriving explicit expressions for the bounds in \eqref{eq:Y} and \eqref{eq:Z} are it can be checked rigorously by a computer using interval arithmetic.

\subsection{Norms and spaces}\label{sec:nands}

We are interested in solutions of PDEs whose spectral representation have coefficients with very 
rapid decay.  To be more precise let us consider  
\begin{equation}\label{eq:fouriercoeffspace}
\ell^{1}_{\nu} = \left\{a = (a_k)_{k\geq 0}, a_k\in\mathbb{R}: \quad |a|_{\nu} \bydef |a_{0}| + 2\displaystyle\sum_{k = 1}|a_k|\nu^k<\infty\right\}.
\end{equation}
Denote the induced operator norm by $|\cdot|_{\ell^{1}_{\nu}}$. Note that if $\nu > 1$ and $a \in \ell^{1}_{\nu}$ the function 
\[
u(x) = a_0 + 2 \sum_{k=1}^\infty a_k \cos(k x),
\]
is real analytic and extends to a periodic and analytic function on the complex strip of 
width $\log(\nu)$.  Moreover we have that 
\[
\| u \|_{C^0([0, 2 \pi])} := \sup_{x \in [0, 2 \pi]} |u(x)| \leq |a|_\nu,
\] 
i.e. the $\ell^{1}_{\nu}$ norm provides bounds on the supremum norm of the corresponding analytic 
function.  

$\ell^{1}_{\nu}$ is a Banach algebra under the discrete convolution operation $a\ast b$ given by 
\begin{equation}
(a\ast b)_{k} = \displaystyle\sum_{\stackrel{k_1+k_2 = k}{k_{i}\in\mathbb{Z}}}a_{|k_1|}b_{|k_2|}.
\end{equation}
with $a,b\in \ell^{1}_{\nu}$. For later use we define the notation $a^{\ast n}$ for $\underbrace{a\ast\cdots\ast a}_{n\hspace{5pt}\text{times}}$.

When we consider unstable manifolds for PDEs we are interested in analytic functions taking their values in $\ell^{1}_{\nu}$ as
just defined.  Such functions have convergent power series representations of the form
\begin{equation} \label{eq:FourierTaylorSeries}
P(\theta, x) = \sum_{|m|=0}^\infty \left( p_{m0}   + 2  \sum_{n =1}^\infty p_{m k} \cos(k x) \right) \theta^m, 
\end{equation}
where $m = (m_1, \ldots, m_d) \in \mathbb{N}^d$ is a $d$-dimensional multi-index, 
$\theta = (\theta_1, \ldots, \theta_d) \in \mathbb{C}^d$ 
and $\{p_{mk}\}_{m \in \mathbb{N}^d  k \in \mathbb{N}}$
is a sequence of \textit{Fourier-Taylor} coefficients (more specifically \textit{cosine-Taylor} coefficients in this case).  

We shorten this notation and write $\{p_m \}_{m \in \mathbb{N}^d}$, where for each $m \in \mathbb{N}^d$ we have
$p_m \in \ell^{1}_{\nu}$.
Then we are led to consider the multi-sequence space
\begin{equation}\label{eq:powercoeffspace}
X^{\nu,d} = \left\{p = (p_m)_{m\in\mathbb{N}^{d}}: p_m \in \ell^{1}_{\nu}
\text{ and } \|p\|_{\nu}\bydef \displaystyle\sum_{|m| = 0}^{\infty}|p_{m}|_{\nu}<\infty\right\},
\end{equation}
of power series coefficients in \eqref{eq:pow_ans}. We will drop the superscript $d$ whenever the dimension of the unstable manifold at hand is clear from the context.  We note that if $p \in X^{\nu}$ then the function 
$P(\theta, x)$ defined in Equation \eqref{eq:FourierTaylorSeries} is periodic and analytic in the 
variable $x$ on the complex strip 
with width $\log(\nu)$, and is analytic on the $d$-dimensional unit polydisk
$\mathbb{B}_1 \subset \mathbb{C}^d$ given by 
\[
\mathbb{B}_1 := \left\{ \theta = (\theta_1, \ldots, \theta_d) \in \mathbb{C}^d : \max_{1 \leq j \leq d} | \theta_j| < 1 \right\}.
\]
Moreover we have that 
\[
\sup_{\theta \in D} \sup_{x \in [0, 2 \pi]} |P(\theta, x)| \leq \| p \|_\nu,
\]
i.e. the norm on $X^\nu$ bounds the supremum norm of $P$.

The space $X^{\nu}$ inherits a Banach algebra structure from the  multiplication operator in the function space
representation. 

\begin{definition}\label{def:Fourier_Taylor}

Let two sequences $p,q\in X^{\nu,d}$ be given. Define $\ast_{TF}:X^{\nu,d}\times X^{\nu,d}\to X^{\nu,d}$ by 
\begin{equation}\label{eq:TF_convolution}
(p\ast_{TF}q)_{m} = \displaystyle\sum_{l\preceq m}p_l\ast q_{m-l},
\end{equation}
where $l\preceq m$ means $l_{i}\leq m_{i}$ for all $i = 1,\ldots,d$. Set $p^{\ast_{TF}n}\bydef \underbrace{p\ast_{TF}\cdots \ast_{TF}p}_{n\hspace{5pt}\text{times}}$.
\end{definition}
The well-definedness of this operation follows from the following lemma.
\begin{lemma}\label{lem:convlemma}
Let $p,q\in X^{\nu,d}$ be given. Then
\begin{equation}\label{lem:TFBanach}
\|p\ast_{TF}q\|_{\nu}\leq \|p\|_{\nu}\|q\|_{\nu}.
\end{equation}
In particular $(X^{\nu},\ast_{TF})$ is a Banach algebra.
\end{lemma} 
Recall that if $F \colon X \to Y$ is a mapping between Banach 
spaces and $x \in X$ then we say that $F$ is 
\textit{Fr\'{e}chet differentiable at} $x$ if there exists a 
bounded linear operator $A \colon X \to Y$ so that 
\[
\lim_{\| h \| \to 0} \frac{\| F(x + h) - F(x) - A h\|_Y}{\|h\|_X} = 0. 
\] 
Recall that the operator $A$, if it exists, is unique.  
When there is such an $A$ we write $DF(x) := A$.
Since $\ell_\nu^1$ and $X^\nu$ are both commutative 
Banach algebras we recall the following general facts
from the calculus of Banach algebras.  Let $(X, *)$
be a commutative Banach algebra. 
It is a straightforward exercise to prove the following:
\begin{itemize}
\item For any fixed $a \in X$ the map $L \colon X \to X$
defined by $L(x) = a * x$ is a bounded linear operator, 
hence it is Fr\'{e}chet differentiable with  $DL \, h = a*h$
for all $h \in X$.
\item The monomial operator 
$F \colon X \to X$ defined by $F(x) = x*x$ is Fr\'{e}chet differentiable 
with $DF(x)h = 2 x*h$ for any $h \in X$.  
\item Applying this rule inductively gives that the nonlinear 
map $G \colon X \to X$ defined by $G(x) = x^{*_n}$
is Fr\'{e}chet differentiable with $DG(x) h = n x^{*_{n-1}} h$
for all $h \in X$.
\end{itemize}

\subsection{Computer assisted verification of the unstable eigenvalue count for a bounded perturbation 
of an eventually diagonal linear operator}\label{subsec:morse_index}

In the sequel we are interested in counting the number of unstable eigenvalues of 
certain linear operators which arise as small, infinite dimensional, perturbations of 
some finite dimensional matrices.  The following spectral perturbation lemma
is formulated in a fashion which is especially well suited to our computational needs.
Similar results have appeared in \cite{MR727816, MR3338319, kotParm}.
See also Remark \ref{rem:CAP_PDE}.
The approach described in this section takes rather explicit advantage of 
the sequence space structure of the problem.  In particular we study a
class of linear operators which have a ``infinite matrix'' representation.  
First some notation.

Suppose that $A, Q, Q^{-1} \colon \ell^{1}_{\nu} \to \ell^{1}_{\nu}$  are bounded linear operators.  Assume that 
$A$ is compact and that $\{\lambda_j \}_{j = 0}^\infty$ are the eigenvalues of $A$.
Suppose that for some $m \geq 0$ the eigenvalues satisfy
\[
  0 <   \mbox{{ real}}(\lambda_m) \leq \ldots \leq   \mbox{{real}}(\lambda_0),
\]
i.e. that there are $m+1$ unstable eigenvalues.  Assume that for all $j \geq m+1$ we have
\[
 \mbox{{real}}(\lambda_j) < 0,
\]
i.e. the remaining eigenvalues are stable.

Then $A$ has no eigenvalues on the imaginary axis, i.e. $A$ is hyperbolic.
Moreover we assume that the stable spectrum of $A$ 
is contained in some cone in the left half plane.  More precisely, suppose that there is $\mu_0 > 0$ so that 
\[
\mu_0 := \sup_{j \geq 0} \sqrt{1 + \left( \frac{\mbox{imag}(\lambda_j)}{\mbox{real}(\lambda_j)} \right)^2} < \infty.
\]
Note that, since $A$ is compact, the $\lambda_j$ accumulate only at zero and the spectrum of 
$A$ is comprised of the union of these eigenvalues and the origin in $\mathbb{C}$.

Now suppose that $A$ factors as 
\[
A = Q \Sigma Q^{-1},
\]
where for all $h \in \ell^{1}_{\nu}$ we define
\[
(\Sigma h)_k = \lambda_k h_k,
\]
i.e. suppose that $A$ is diagonalizable.
Note that $\Sigma$ is a compact operator.  
Consider the operator $\Sigma^{-1}$ given by 
\[
(\Sigma^{-1} h)_k = \frac{h_k}{\lambda_k},  
\]
for $k \geq 0$.  
The operator is formally well defined as the assumption that all the $\lambda_j$ have non-zero real part 
implies in particular that $\lambda_j \neq 0$ for all $j \geq 0$.  Moreover the 
operator $\Sigma^{-1}$ has exactly $m$ unstable eigenvalues $1/\lambda_j$ for $0 \leq j \leq m$, and 
the stable spectrum is contained in the same cone as the stable spectrum of $A$. 
 Note however that $\Sigma^{-1}$ need not be a bounded linear operator 
on $\ell^{1}_{\nu}$.  Nevertheless one checks that 
\[
\Sigma \Sigma^{-1} = \mbox{I} \quad \quad \quad \mbox{and} \quad \quad \quad 
\Sigma^{-1} \Sigma = \mbox{I},
\]  
on $\ell^{1}_{\nu}$.
In the applications below $\Sigma^{-1}$ will be a densely defined operator on $\ell^{1}_{\nu}$.

Consider now the operator
\[
B = Q \Sigma^{-1} Q^{-1}.
\]
$B$ is formally well-defined (in fact has the same domain as $\Sigma^{-1}$) and has
\[
A B = \mbox{I},   \quad \quad \mbox{and} \quad \quad B A  = \mbox{I},
\]
on $\ell^{1}_{\nu}$.  
Moreover, if $\Sigma^{-1}$ is densely defined so is $B$.
The eigenvalues of $B$ are precisely $\frac{1}{\lambda_k}$
and in particular $B$ has exactly the $m+1$ unstable eigenvalues 
$\frac{1}{\lambda_j}$ for $0 \leq j \leq m$.  We are interested in bounded perturbations of 
$B$, and have the following Lemma.

\begin{lemma} \label{lem:eigCount}
Suppose that $A, Q, Q^{-1} \colon \ell^{1}_{\nu} \to \ell^{1}_{\nu}$, $\{\lambda_j\}_{j=0}^\infty \subset \mathbb{C}$
and $\mu_0 > 0$ are as discussed above, and that 
$B = Q \Sigma^{-1} Q^{-1}$ is a (possibly only densely defined) linear operator on $\ell^{1}_{\nu}$.
Let $H \colon \ell^{1}_{\nu} \to \ell^{1}_{\nu}$ be a bounded linear operator 
and let $M$ be the (densely defined) linear operator 
\[
M = B + H.
\]
Assume that $\epsilon > 0$ is a positive real number with 
\[
\|\mbox{I} - AM \|_{B\left(\ell^{1}_{\nu} \right)} \leq \epsilon,
\]
and
\[
\| Q \|_{B\left(\ell^{1}_{\nu} \right)} \|Q^{-1}\|_{B\left(\ell^{1}_{\nu} \right)} \mu_0 \epsilon < 1.
\]
Then $M$ has exactly $m$ unstable eigenvalues.  
\end{lemma}

\begin{proof}
The result and its proof are similar to Lemma E.1 of \cite{kotParm}.
Indeed, we will construct a homotopy from $A$ to $B$ just as
in Lemma E.1.  Repeating the argument of Step 1 of the 
proof of Lemma E.1,  one sees that
$A$ and $B$ have the same Morse index
as soon as we can show that no eigenvalues cross the imaginary axis 
during the homotopy.  

We begin by noting that for all $\mu \in \mathbb{R}$ the operator 
\[
B - i \mu \mbox{I},
\]
is boundedly invertible, as $i \mu$ is not in the spectrum of $B$.
Similarily, we have that the operator 
\[
\mbox{I} - \mu i A = Q (\mbox{I} - \mu i \Sigma ) Q^{-1},  
\]
is boundedly invertible.  To see this note that
\[
(\mbox{I} - \mu i A)^{-1} = Q (\mbox{I} - \mu i \Sigma)^{-1} Q^{-1},
\]
where 
\[
[(\mbox{I} - \mu i \Sigma)^{-1} h]_k = \frac{1}{1 - \mu i \lambda_k} h_k
\]
for all $k \geq 0$, and since $\lambda_k$ is never purely imaginary this denominator is 
never zero.  Indeed for each $j$ the worst case scenario is that $\mu = \mbox{imag}(\lambda_j^{-1})$, 
so that 
\[
\sup_{j \geq 0} \left|  \frac{1}{1 - \mu i \lambda_j} \right| \leq 
\sup_{j \geq 0} \left|\frac{\lambda^{-1}_j}{\lambda_j^{-1} - i \mu}   \right|
\]
\begin{eqnarray*}
&\leq& \sup_{j \geq 0} \left|\frac{\lambda^{-1}_j}{\lambda_j^{-1} - i\,   \mbox{imag}(\lambda_j^{-1})}   \right| \\
&=& \sup_{j \geq 0} \left|  \frac{\lambda_j^{-1}}{\mbox{real}(\lambda_j^{-1})}   \right| \\
&=& \sup_{j \geq 0} \sqrt{1 + \left( \frac{\mbox{imag}(\lambda_j^{-1})}{\mbox{real}(\lambda_j^{-1})} \right)^2} \\
&=& \sup_{j \geq 0} \sqrt{1 + \left( \frac{\mbox{imag}(\lambda_j)}{\mbox{real}(\lambda_j)} \right)^2} \\
&=& \mu_0,
\end{eqnarray*}
as $\mbox{Arg}(\lambda_j^{-1}) = -\mbox{Arg}(\lambda_j)$.  From this we obtain that 
\[
\|(\mbox{I} - \mu i A)^{-1} \|_{B\left(\ell^{1}_{\nu} \right)} \leq \| Q \|_{B\left(\ell^{1}_{\nu} \right)} \| Q^{-1}\|_{B\left(\ell^{1}_{\nu} \right)}
\sup_{j \geq 0} \left| \frac{1}{1 - i \mu \lambda_j} \right|  \leq 
\| Q \|_{B\left(\ell^{1}_{\nu} \right)} \| Q^{-1}\|_{B\left(\ell^{1}_{\nu} \right)} \mu_0.
\]
Note that 
\[
\| A H \| = \| A (M - B)\| = \| AM - AB \| = \| \mbox{I} - AM \| \leq \epsilon < 1,
\]
by hypothesis.

We now consider the homotopy  
\[
C_t = B + tH,
\]
for $t \in [0, 1]$ and note that $C_0 = B$ and $C_1 = M$.
Again, we take $\mu \in \mathbb{R}$ and consider the resolvent operator 
\begin{eqnarray*}
C_t - i \mu \mbox{I} &=& B - i \mu \mbox{I} + t H \\
&=& (B - i \mu \mbox{I}) \left[ \mbox{I} + t(B - i \mu \mbox{I} )^{-1}   H \right] \\
&=& (B - i \mu \mbox{I}) \left[ \mbox{I} + t(B - i \mu \mbox{I} )^{-1}  B A  H \right] \\
&=& (B - i \mu \mbox{I}) \left[ \mbox{I} + t(\mbox{I} - i \mu A )^{-1}   A  H \right]. \\
\end{eqnarray*}
Note that for all $t \in [0, 1]$ we have
\[
 \|t(\mbox{I} - i \mu A )^{-1}   A  H\| \leq \| Q \|_{B\left(\ell^{1}_{\nu} \right)} \| Q^{-1}\|_{B\left(\ell^{1}_{\nu} \right)} \mu_0 \|A H \| 
 \leq \| Q \|_{B\left(\ell^{1}_{\nu} \right)} \| Q^{-1}\|_{B\left(\ell^{1}_{\nu} \right)} \mu_0 \epsilon < 1,
\]
by hypothesis.  By the Neumann theorem, $C_t - i \mu \mbox{I}$
is boundedly invertible for all $\mu \in \mathbb{R}$ and all $t \in [0,1]$.  Then 
the resolvent is boundedly invertible throughout the homotopy, which implies
that no eigenvalues cross the imaginary axis.
Then the number of unstable eigenvalues is constant throughout the homotopy, i.e. 
$B$ and $M$ have exactly $m$ unstable eigenvalues as claimed.
\end{proof}


\section{Parameterization Method for unstable manifolds of parabolic PDE} \label{sec:ParmMethod}
Since its introduction in \cite{param1, param2, param3, whiskeredtori1, whiskeredtori2},
a small industry has grown up around the Parameterization Method and its applications.
A proper review of the literature would make an excellent subject for an entire manuscript,
and is certainly beyond the scope of the present work.  A wonderful survey of the subject, 
with many applications and thorough discussion of the literature, 
is found in the book of \cite{mamotreto}.  

The modest aim of the present discussion is to give the reader the flavor of the
activity in this area.  More importantly, we hope to indicate that the 
Parameterization Method is a much more general tool than the present work, 
viewed in isolation, would suggest.  
Here follows a brief (and by no means definitive) survey of some of 
problems which have been addressed using the method.
The list is organized by topic with corresponding citations.
The interested reader will be lead to many additional techniques and applications
by consulting the references of the works cited here.

\begin{itemize}
\item \textit{Stable/unstable manifolds for non-resonant fixed points/equilibria:}
theory for maps on Banach spaces \cite{param1, param2}, 
parabolic fixed points \cite{MR2276478},
numerical implementation for maps and ODEs 
\cite{jayvortex, MR2728178, MR2728178, MR3437754, alexADmanuscript},
validated numerical methods for maps and ODEs 
\cite{MR3068557, parmResPaper, JBJPJayKons, parmResPaper}.
\item \textit{Stable/unstable bundles and manifolds for invariant tori:}
discrete time 
\cite{MR2289544, MR2240743, MR3309008, MR3082311, MR2299977, MR2507954},
Hamiltonian ODEs and PDEs \cite{whiskeredtori3, MR2505176}
\item \textit{Stable/unstable manifolds of periodic orbits for ODEs:}
theory and numerical implementation \cite{param3, MR3118249, MR2551254, paramperiodic}.
\item \textit{KAM without action angle variables:} \cite{MR3118249}
invariant tori for symplectic maps \cite{MR2122688}, 
invariant tori for non-autonomous Hamiltonian systems \cite{MR3396853},
invariant tori in dissipative systems \cite{MR3062760},
mixed stability invariant manifolds associated with fixed points in 
symplectic/volume preserving maps \cite{MR2966749}.
\item \textit{Map Lattices:} stable/unstable manifolds for invariant tori \cite{MR3353644},
almost-periodic breathers and their stable/unstable manifolds \cite{MR3147408}.
\item \textit{Invariant tori for state dependent delay equations:} $C^k$/hyperbolic case \cite{delayParmI}, 
Analytic/KAM case \cite{delayParmII}.
\item \textit{Invariant manifolds for dissipative infinite dimensional dynamical 
systems:} parabolic PDEs \cite{jpJordiMarcioRafaPaper}, compact maps
\cite{kotParm}.
\end{itemize}

\subsection{A-priori existence for solutions of Equation \eqref{eq:func_eq}:
a non-resonant unstable manifold theorem for parabolic PDE} \label{sec:aPrioriExist}
Let $X$ be a Banach algebra and let $A$ be a closed, densely defined,
with compact resolvent.  Then $A$ generates a compact semigroup, 
which we denote by $e^{At}$, $t \geq 0$.  An explicit formula for the 
exponential can be obtained as a line integral of the resolvent 
operator in the complex plane.  While we make no explicit use of this  
representation in the present work, it is important to us that 
an estimate of the form 
\[
\|e^{A t} \|_{B(X)} \leq M e^{\mu_* t},
\]
can be obtained from this line integral representation. 
The explicit values of the constants will not matter to us in the sequel.
However we will exploit the fact that if $A$ is sectorial with spectrum 
in the open left half plane then $\mu_* < 0$.  Such an operator 
will be called dissipative.  The material above is standard in 
the theory of analytic semi-groups and parabolic PDEs. See for example 
\cite{MR1637218, MR1721989}.

Now let $G \colon X \to X$ be a Fr\'{e}chet differentiable map.
In fact we will be interested in the convergence of certain formal series,
so we assume in addition that $G$ is analytic.
Consider the differential equation
\begin{equation}\label{eq:BanachODE}
x' = Ax + G(x) 
\end{equation}
Suppose that $\tilde x \in X$ is a hyperbolic stationary solution
of Equation \eqref{eq:BanachODE},
i.e. that $ A + DG(\tilde x)$ has no eigenvalues on the imaginary axis.
Since $DG(\tilde x)$ is a bounded linear operator the operator $ A + DG(\tilde x)$
remains sectorial, and again generates a compact semigroup.  
Then $A + DG(\tilde x)$ has at most finitely many unstable eigenvalues of 
finite multiplicity.   
We denote these eigenvalues by $\lambda_1, \ldots, \lambda_d$
and suppose for the sake of simplicity (this is the case studied in the 
sequel) that they are real and distinct, i.e. of multiplicity one.  
Then there are eigenvectors $\xi_1, \ldots, \xi_d$ which we choose to have 
unit norm (we will rescale them by explicit constants below).

After an affine change of variables we can arrange that  
$\tilde x = 0$ and that the equation becomes 
\begin{equation}\label{eq:abstractODE}
x' = Ax + N(x) \bydef F(x),
\end{equation}
with
\[
A = \left(
\begin{array}{cc}
A_u & 0 \\
0 & A_s
\end{array}
\right),
\]
where 
\[
A_u = 
 \left(
\begin{array}{ccc}
\lambda_1 & \ldots & 0 \\
\vdots & \ddots & \vdots \\
0 & \ldots & \lambda_d
\end{array}
\right)
\]
and $A_s$ is a dissipative operator, i.e. 
there are $M, \mu_* > 0$ so that 
\[
\| e^{A_s t} \|_{B(X)} \leq M e^{- \mu_* t},
\]
for all $t \geq 0$.  
Moreover the semigroup $e^{A_s t}$ is 
compact. The function $N$, is analytic in a neighborhood
of the origin and is zero to second order, and we write
\begin{equation}\label{eq:nonlinearity}
N(x) = \sum_{j = 2}^\infty n_j x^{* j},
\end{equation}
with coefficients $n_j \in X$. Since $N$ is analytic on 
the disk there exists an 
$R > 0$ so that 
\[
\sum_{j = 2}^{\infty} \| n_j \| R^{|m|} < \infty.
\]

\begin{definition}[Resonance of order $m$] \label{def:nonRes}
We say that the complex numbers $\lambda_1, \ldots, \lambda_d$ have a 
resonance of order $(m_1, \ldots, m_d) = m \in \mathbb{N}^d$ if 
\[
m_1 \lambda_1 + \ldots + m_d \lambda_d = \lambda_j,
\]
for some $1 \leq j \leq d$ and some $|m| \geq 2$.  
We say that $\lambda_1, \ldots, \lambda_d$ are non-resonant if 
there is no resonance of order $m$ for any order $|m| \geq 2$.
\end{definition}

If we consider $\lambda_1, \ldots, \lambda_d$ a finite collection of unstable eigenvalues,
then there are only finitely many opportunities for resonances between these (as for $|m|$ 
large enough the product on the left has magnitude larger than any unstable
eigenvalue).  Then, despite first impressions, Definition \ref{def:nonRes} imposes
only a finite number of constraints.

Define the linear approximation 
\[
P_1(\theta_1, \ldots, \theta_d) = \tilde x + s_1 \xi_1 \theta_1 + \ldots + s_d \xi_d \theta_d.
\]
where the $s_1, \ldots, s_d$ are arbitrary non-zero real numbers (these 
numbers can be thought of as ``tuning'' the scalings of the eigenvectors,
as we choose the $\xi_j$ to have unit norm).
We have the following unstable manifold theorem.

\begin{theorem}[Existence of a conjugating chart map for the unstable manifold] \label{thm:smt}
Suppose that the unstable spectrum of $DF(\tilde x)$ consists of only the 
unstable eigenvalues $\lambda_1, \ldots, \lambda_d$.  Assume that
each of these is real and has finite multiplicity one. 
Assume in addition that $\lambda_1, \ldots, \lambda_d$ are non-resonant in 
the sense of Definition \ref{def:nonRes}.
Then there is a $\delta > 0$ so that if  
\[
\max_{1 \leq j \leq d} |s_j| \leq \delta,
\]
then there exists a unique
$P \colon B_1 \to X$ satisfying the first order constraints
\begin{equation}\label{eq:lin_BA}
P(0) = \tilde x, 
\quad \quad \quad \mbox{and} \quad \quad \quad 
\frac{\partial}{\partial \theta_j} P(0) = s_j \xi_j,
\end{equation}
and having that $P$ is a solution of 
the invariance equation
\begin{equation}\label{eq:func_eqBA}
F(P(\theta)) = DP(\theta)A_u \theta
\end{equation}
on $B_1$. In particular $P[B_1]$ is a local 
unstable manifold at $\tilde x$.
\end{theorem}

To solve \eqref{eq:func_eqBA} we introduce some notation. Let $\pi_u, \pi_s \colon X \to X$ be the spectral projections associated 
with $A_u$ and $A_s$ and let $X_u := \pi_u(X)$, $X_s = \pi_s(X)$.
Then for any $x \in X$ we can write $x = (x_1, x_2)$ with 
$x_1 = \pi_u(x) \in X_u$ and $x_2 = \pi_s(x) \in X_s$.
Indeed, we can further decompose $X_u$ by noting that 
each of the eigenvalues $\lambda_j$, $1 \leq j \leq d$ has
an associated spectral projection operator $\pi_j \colon X \to X$.
Define the linear subspaces $X_j = \pi_j(x)$ and 
note that $X_u = X_1 \oplus \ldots \oplus X_d$.  Moreover, 
since $\xi_j$ spans $X_j$ we have that each 
$x_j \in X_j$ can be written uniquely as $x_j = c_j \xi_j$
for some scalar $c_j$.

Define 
\[
\mu^* = \min_{1 \leq j \leq d} \mbox{real}(\lambda_j), 
\]
and note that this is $\mu^* = \min_{1 \leq j \leq d} \lambda_j$
as the unstable eigenvalues are assumed to be real.

We look for a solution of Equation \eqref{eq:func_eqBA} in the form 
\[
P(\theta) = P_1(\theta) + H(\theta),
\]
where $H(0) = \partial H/\partial \theta_j(0) = 0$ for $1 \leq j \leq d$.
Note that in this context the left hand side of 
Equation \eqref{eq:func_eqBA} becomes
\begin{eqnarray*}
F(P) &=& F(P_1 + H) \\
 &=& A P_1 + AH + N(P_1 + H) 
\end{eqnarray*}
while the right hand side is 
\begin{eqnarray*}
DP(\theta) A_u \theta  &=& DP_1(\theta) A_u \theta + DH(\theta) A_u \theta \\
&=& \lambda_1 \theta_1 s_1 \xi_1 + \ldots + \lambda_d \theta_d s_d \xi_d + DH(\theta) A_u \theta.
\end{eqnarray*}
We note that 
\begin{eqnarray*}
A P_1(\theta) &=&  
s_1 \theta_1 A \xi_1 + \ldots + s_d \theta_d A \xi_d \\
&=& \lambda_1 \theta_1 s_1 \xi_1 + \ldots + \lambda_d \theta_d s_d \xi_d,
\end{eqnarray*}
as $\lambda_1, \xi_j$ are eigenvalue/eigenvector pair for $A$.
After cancelation of these terms on the left and the right, Equation \eqref{eq:func_eqBA}
becomes
\begin{equation}\label{eq:prePhiEqn}
DH(\theta) A_u \theta - AH(\theta) = N(P_1(\theta) + H(\theta)).
\end{equation}
The left hand side of this expression defines a boundedly 
invertible linear operator on a suitable space of functions,
as the next lemma shows. For $P \colon B_1 \to X$ analytic 
define the norm
\[
\| P \|_1 := \sum_{|m| = 0}^\infty \| p_m \|
\] 
where $p_m \in X$ for each $m \in \mathbb{N}^d$.
Note that 
\[
\sup_{\theta \in B_1} \| P(\theta)\| \leq \| P \|_1,
\]
where the norm on the left is the norm on $X$ and the inequality 
holds even in the case that one is infinite.  
We employ  the $\| \cdot \|_1$ norm below (and throughout the 
remainder of the paper) in spite of the 
fact that this is less fine than the supremum norm, due to the fact that 
it makes numerical calculations and some formal manipulations easier.
Define 
\[
\mathcal{H} = \left\{ H  \in C^\omega(B_1, X) \, | \, H(0) = \frac{\partial H}{\partial \theta_j} (0)
 = 0, \mbox{ and } \| H \|_1 < \infty
 \right\}.
 \]

\begin{lemma}
Suppose that the unstable eigenvalues $\lambda_1, \ldots, \lambda_d$ 
are non-resonant in the sense of Definition \ref{def:nonRes}.  Then the 
linear operator given by 
\[
\mathcal{L}[H](\theta) = DH(\theta) A_u \theta - A H(\theta),
\]
is boundedly invertible on $\mathcal{H}$.  
\end{lemma}

\begin{proof}
Let $E \in \mathcal{H}$.  Taking the spectral projections we write
\[
E(\theta) = 
\left(
\begin{array}{c}
E_u(\theta) \\
E_s(\theta) \\
\end{array}
\right),
\]
where $E_u = E_1 + \ldots E_d$.
Consider the projected equations
\begin{equation}\label{eq:L1_def} 
DH_u(\theta) A_u \theta + A_u H_u(\theta) = E_u(\theta), 
\end{equation}
and 
\begin{equation}\label{eq:L2_def} 
DH_s(\theta) A_u \theta + A_s H_s(\theta) = E_s(\theta).
\end{equation}
We begin by solving Equation \eqref{eq:L1_def} term by term
in the sense of power series.
    
Since $E$ is analytic we can write
\[
E(\theta) = \sum_{|m| = 2}^\infty e_m \theta^m,
\]
for some $e_m \in X$, where the series converges absolutely 
and uniformly for all $\theta \in B_1$, and in fact 
\[
\sum_{|m| = 2}^\infty \|e_m\| = \| E \|_1 < \infty,
\] 
as $E \in \mathcal{H}$.
Taking projections, we write 
\[
E_u(\theta) = E_1(\theta) + \ldots + E_d(\theta),
\]
where
\[
E_j(\theta) = \sum_{|m| = 2}^\infty e_m^j \theta^m
= \sum_{|m| = 2}^\infty b_{m}^j \xi_j \theta^m,
\]
with  $e_{m}^j = \pi_j(e_m)$ and $e_{m}^j = b_m^j \xi_j$
for some unique scalars $b_m^j$.

Proceeding formally, we look for $H_u \colon B_1 \to X_u$ in the form 
\[
H_u(\theta) = H_1(\theta) + \ldots + H_d(\theta),
\]
where $H_j \colon B_1 \to X_j$ is given by 
\[
H_{j}(\theta) = \sum_{|m| = 0}^\infty c_{m}^j \xi_j \theta^m,
\]
for some unknowns scalar coefficients $c_{m}^j$.
Projecting Equation \eqref{eq:L1_def} onto $X_j$ for $1 \leq j \leq d$, and 
noting that $A_u$ is a diagonal matrix leads to the 
equations
\[
D H_j(\theta) A_u \theta + \lambda_j H_j(\theta) = E_j(\theta),
\]
so, upon making the power series substitutions, we have that 
\[
\sum_{|m| = 2}^\infty 
(m_1 \lambda_1 +\ldots m_d \lambda_d - \lambda_j) c_{m}^j \xi_j  \theta^m= 
\sum_{|m| = 2}^\infty b_m^j \xi_j \theta^m.
\]
Matching like powers of $\theta$ leads to 
\[
(m_1 \lambda_1 +\ldots m_d \lambda_d - \lambda_j) c_{m}^j \xi_j
=  b_m^j \xi_j,
\]
from which we conclude that 
\[
c_{m}^j = \frac{1}{m_1 \lambda_1 +\ldots m_d \lambda_d - \lambda_j} b_{m}^j,
\]
for all $1 \leq j \leq d$ and $|m| \geq 2$.

Motivated by the discussion above we define the linear solution operators
\[
\mathfrak{L}_j^{-1}[E](\theta) = 
\sum_{|m| = 2}^\infty 
\frac{1}{m_1 \lambda_1 + \ldots + m_d \lambda_d - \lambda_j} 
 \,  \pi_j(e_m)  \theta^m,  \quad \quad \quad 1 \leq j \leq d,
\]
for $1 \leq j \leq d$.  Note that that $\mathfrak{L}_j^{-1}$ are well defined as
the $\lambda_1, \ldots, \lambda_d$ are non-resonant.  Defining 
\[
C_j = \max_{|m| \geq 2} | m_1 \lambda_1 + \ldots + m_d \lambda_d - \lambda_j |^{-1},
\] 
we see that the solution operators are bounded on $\mathcal{H}$ as
\begin{eqnarray*}
\|\mathfrak{L}_j^{-1}[E_u](\theta)\|_1 &=& \sum_{|m| = 2}^\infty 
\frac{1}{|m_1 \lambda_1 + \ldots + m_d \lambda_d - \lambda_j|} 
 \| \pi_j(e_m) \| \\
&\leq& C_j \sum_{|m| = 2}^\infty\| \pi_j\|  \| e_m \| \\
&\leq& C_j \|\pi_j\| \| E\|_1,
\end{eqnarray*}
which is bounded due to the fact that  
the spectral projections are bounded linear operators and $E \in \mathcal{H}$.
Now let $H_j := \mathfrak{L}_j^{-1}[E]$ for $1 \leq j \leq d$, 
and $H_u = H_1 + \ldots + H_d$.  
Then $DH_j \in \mathcal{H}$ for each $1 \leq j \leq d$, and 
working the argument backwards shows that $H_u$ 
is indeed a  solution of Equation  \eqref{eq:L1_def}.

In order to solve Equation \eqref{eq:L2_def}, 
consider the change of variables
\[
\theta \to e^{\lambda t} \theta,
\]
with $\theta \in B_1$ fixed, 
and define 
\[
x(t) = H_s(e^{\lambda t} \theta),
\quad \quad \quad \mbox{and} \quad \quad \quad 
p(t) = E_s(e^{\lambda t}  \theta).
\]
Suppose that $x(t)$ is a solution of the differential equation 
\begin{equation}\label{eq:auxODE}
x' - A_s x = p,
\end{equation}
for all $t \leq 0$. Then $x(0)$ is a solution Equation 
\eqref{eq:L2_def}.  
But solutions of Equation \eqref{eq:auxODE} are
given by Duhamel's formula.
More precisely, we begin by multiplying both 
sides by $e^{-A_s t}$ and integrating form $t_0$ to $t_1$ to obtain 
\begin{equation}\label{eq:intForm1}
e^{-A_s t_1} x(t_1) - e^{-A_s t_0} x(t_0) = 
\int_{t_0}^{t_1} e^{-A_s t} E_s(e^{\lambda t} \theta) \, dt.
\end{equation}
Assuming that $H_s \in \mathcal{H}$ (so that $H$ is zero to second order)
we have that 
\begin{eqnarray*}
\lim_{t_0 \to -\infty} e^{-A_s t_0} x(t)
&=& \lim_{t_0 \to \infty} e^{A_s t_0} x(-t_0) \\
&=& \lim_{t_0 \to \infty} e^{A_s t_0} H_s(e^{-\lambda t_0} \theta) \\
&=& 0
\end{eqnarray*}
as
\begin{eqnarray*}
\|e^{A_s t_0} H_s(e^{-\lambda t_0} \theta) \| &\leq&
\| e^{A_s t_0} \| \| H_s(e^{-\lambda t_0} \theta \| \\ 
&\leq& M e^{-\mu_*t_0} \left(e^{-\mu^* t_0}\right)^2 \| H_s \| \\
&\leq& M e^{-(\mu_* + 2 \mu^*) t_0} \| H_s \|
\end{eqnarray*}
and $\| H_s \| < \infty$.
Then taking $t_0 \to -\infty$ in Equation \eqref{eq:intForm1}
and $t_1 = 0$ gives 
\[
x(0) = \int_{-\infty}^{0} e^{-A_s t} E_s(e^{\lambda t} \theta) \, dt = 
 \int_{0}^{\infty} e^{A_s t} E_s(e^{-\lambda t} \theta) \, dt,
\]
after switching the limits of integration and changing 
$t \to -t$.

Motivated by this discussion we define the linear solution operator 
\[
\mathfrak{L}^{-1}_s[E_s](\theta) := \int_0^\infty e^{A_s t} \pi_s\left[E(e^{-\lambda t} \theta)\right] \, dt.
\]
Let $H_s := \mathfrak{L}^{-1}_s[E_s]$, and 
note that
\[
\| H_s \| \leq \frac{M}{2 \mu^* + \mu_*} \| E_s \|,
\]
as the integrand satisfies
\[
\| e^{A_s t} E_s(e^{-\lambda t} \theta)\| 
\leq M e^{-\mu_* t} e^{-2 \mu^* t} \|E_s\|,
\]
i.e. the operator is well defined and bounded.  
Moreover we see that $H_s$ is 
analytic by Morera's Theorem.  To see that $H_s$ 
is zero to second order we 
differentiate under the integral 
and note that $E_s(0) = 0$ and 
$\partial E_s/\partial \theta_j(0) = 0$ for $1 \leq j \leq d$.
To see that $\|H_s \|_1 < \infty$ we expand $E_s$ 
as a power series inside the formula for
$\mathfrak{L}^{-1}_s[E_s]$ and,
after exchanging the sum and the integral, bound  
$\| H_s\|_1$ in terms of $\| E_s \|_1$.  We then check 
by differentiating that $H_s$ so defined solves the desired equation. 
\end{proof}

\begin{proof}[Proof of Theorem \ref{thm:smt}]
Let
\[
\max_{1 \leq j \leq d} |s_j| = s.
\]
Then for any $H \in \mathcal{H}$ with $\|H\| \leq r$ we have
that 
\[
\sup_{\theta \in B_1} \| P_1(\theta) + H(\theta) \| \leq s + r. 
\]
Suppose now that $\lambda_1, \ldots, \lambda_d$ are
non-resonant and choose $s, r > 0$ so that 
\begin{equation}\label{eq:s_plus_r_cond}
s + r < R,
\end{equation}
where $R$ is the radius of convergence of the series
expansion of $N$.
Then the nonlinear operator $\Phi \colon \mathcal{H} \to \mathcal{H}$ 
given by 
\[
\Phi[H](\theta) = \mathfrak{L}^{-1}\left[
N(P_1(\theta) + H(\theta))
\right],
\]
is well defined for any $s, r$ satisfying the Equation 
\eqref{eq:s_plus_r_cond}.  Note also that $N(P_1 + H(\theta))$,
is a composition of analytic functions, hence is analytic on $B_1$.
Then $N \circ (P_1 + H) \in \mathcal{H}$ as is seen by evaluating
$N(P_1(\theta) + H(\theta))$ and its first partials at $\theta = 0$.

The rest of the argument hinges on the fact that 
$H$ is a fixed point of $\Phi$ if and only if
$H$ is a solution of Equation \eqref{eq:prePhiEqn}, 
if and only if $P = P_1 + H$ is a solution of Equation 
\eqref{eq:func_eq}.
In order to establish that $\Phi$ has a fixed point we employ
the contraction mapping theorem.
For the remainder of the argument we suppose that $s, r > 0$
satisfy Equation \eqref{eq:s_plus_r_cond}.

First, note that since $N$ is zero to second order at the origin 
there are $M_1, M_2 > 0$ so that 
\[
\| N(x) \| \leq M_1 \| x\|^2,
\] 
and
\[
\| D N(x) \|_{B(X)} \leq M_2 \| x \|, 
\]
for all $x \in X$ with $\| x \| < R$.  
(Explicit constants can be obtained for example 
by adapting the argument of Lemma $2.5$ of 
\cite{MR3282485}).

Then for any $H \in \mathcal{H}$ with $\|H\| \leq r$ we have 
\begin{eqnarray*}
\|\Phi[H]\| &\leq& \| \mathfrak{L}^{-1} N[P_1 + H]\| \\
&\leq& \|\mathfrak{L}^{-1}\| M_1 (s+r)^2.
\end{eqnarray*}
Also note that if $\|H\| \leq r$ then we 
have that $\Phi$ is Fr\'{e}chet differentiable at $H$
with 
\[
D \Phi[H] v = \mathfrak{L}^{-1} DN(P_1 + H)v,
\]
for $v \in \mathcal{H}$, and the bound
\[
\|D \Phi[H] \| \leq \|\mathfrak{L}^{-1}\| M_2(s+r).
\]
Now choose $H_1, H_2 \in \mathcal{H}$ with
$\|H_1\|, \|H_2\| \leq r$.  We have that  
\begin{eqnarray*}
\| \Phi[H_1] - \Phi[H_2] \| &\leq& \sup_{H \leq r} \| D\Phi[H]\| \|H_1 - H_2\| \\
\| \mathfrak{L}^{-1}\| M_2 (s+r) \|H_1 - H_2 \|.
\end{eqnarray*}  
Suppose that $\delta > 0$ is small enough that 
\[
2 \delta < R,
\]
\[
 \|\mathfrak{L}^{-1}\| M_1 4 \delta^2 \leq \frac{\delta}{2},
\]
and
\[
2 \| \mathfrak{L}^{-1}\| M_2 \delta  < 1.
\]
Then for any choice of $s_1, \ldots, s_d, r > 0$ so that 
\[
s, r \leq \frac{\delta}{2},
\]
we have that $\Phi$ is a contraction on the complete
metric space
\[
U_r = \left\{ H \in \mathcal{H} \, | \, \| H \| \leq r \right\}.
\]
Now the contraction mapping theorem implies that
there exists a unique $\tilde H \in U_r$ so that $\Phi[\tilde H] = \tilde H$.
It follows that 
\[
P(\theta) = P_1(\theta) + \tilde H(\theta),
\]
satisfies Equation \eqref{eq:func_eq}.  Since $P_1$ satisfies 
the constraints of Equation \eqref{eq:linear_constraints} and $\tilde H$ is 
zero to second order at zero, we have that $P$ satisfies 
the first order constraints as well.  Then by Lemma \ref{lem:flowInv},
the image of $P$ is the desired local unstable manifold.
\end{proof}

\begin{remark}[Uniqueness]
{\em
The solution of Equation \eqref{eq:func_eq} obtained in the proof of 
Theorem \ref{thm:smt} is up to the choice of the eigenvectors and their
scalings. In other words we obtain parameterizations of larger or smaller 
local unstable manifolds by choosing larger or smaller scalings 
$s_1, \ldots, s_d$. The non-uniquness is exploited in  numerical computations,
and in theoretical considerations of the decay rates of the Taylor 
coefficients of $H$.
}
\end{remark}

\subsection{Non-uniqueness and Taylor coefficient decay: 
rescaling the unstable eigenvectors} \label{sec:decayRates}

Suppose that $s_j = 1$ for $j = 1,\ldots,d$ in \eqref{eq:lin_BA} are the scalings for the eigenvectors
and that $P \colon B_1 \subset \mathbb{R}^d \to X$ is a corresponding 
analytic solution of Equation \eqref{eq:func_eqBA}.
Now let $s_1, \ldots, s_d > 0$, $s_j\neq 1$, $j = 1,\ldots,d$ and consider the function 
\[
\hat P(\theta_1, \ldots, \theta_d) = P(s_1 \theta_1, \ldots, s_d \theta_d),
\] 
defined for $\theta_j s_j \leq 1$.  Note that 
\[
\hat P(0) = \tilde x,
\]
and that 
\[
\frac{\partial \hat P}{\partial \theta_j} (0) = s_j \frac{\partial  P}{\partial \theta_j}(0) = s_j \xi_j,
\]
so that $\hat P$ satisfies the first order constraints of Equation \eqref{eq:lin_BA}, a scaled version \eqref{eq:linear_constraints}.  Moreover we have that 
\begin{eqnarray*}
F[\hat P(\theta_1, \ldots, \theta_d)] &=& F[P(s_1 \theta_1, \ldots, s_d \theta_d)] \\
&=& \lambda_1 (s_1 \theta_1) 
\frac{\partial}{\partial \theta_1} P(s_1 \theta_1, \ldots, s_d \theta_d)
+ \ldots + 
\lambda_d (s_d \theta_d) 
\frac{\partial}{\partial \theta_d} P(s_1 \theta_1, \ldots, s_d \theta_d) \\
&=& 
 \lambda_1  \theta_1 
\frac{\partial}{\partial \theta_1} \hat P(s_1 \theta_1, \ldots, s_d \theta_d)
+ \ldots + 
\lambda_d \theta_d 
\frac{\partial}{\partial \theta_d} \hat P(s_1 \theta_1, \ldots, s_d \theta_d).
\end{eqnarray*}
In other words $\hat P$ is a solution of Equation \eqref{eq:func_eqBA} corresponding
the the rescaled choice of eigenvectors.  Since the solution of 
Equation \eqref{eq:func_eqBA} is unique up to this choice of scalings we 
see that all solutions of Equation \eqref{eq:func_eqBA} are obtained in 
this manor.

\begin{remark}[Rescaling and Taylor coefficient decay rates]\label{rem:decayRates}
{\em
Consider the power series expansion 
\[
P(\theta) = \sum_{|m| = 0}^\infty p_m \theta^m,
\] 
and a choice of scalings $s_1, \ldots, s_d \neq 0$.  Then 
the rescaled solution $\hat P$ has power series given by  
\begin{eqnarray*}
\hat P(\theta) &=& P(s_1 \theta_1, \ldots, s_d \theta_d) \\
&=& \sum_{|m| = 0}^\infty p_m s_1^{m_1} \ldots s_d^{m_d} \theta^m,
\end{eqnarray*}
i.e. given the Taylor coefficient sequence $\{p_m \}_{|m| = 0}^\infty$ of one solution 
of Equation \eqref{eq:func_eqBA}, the power series coefficients 
$\{\hat p_{m}\}_{|m| = 0}^\infty$ of all other solutions
of Equation \eqref{eq:func_eqBA} are obtained by the transformation
\begin{equation} \label{eq:rescaledCoefficients}
\hat p_{m} = s_1^{m_1} \ldots s_d^{m_d} p_{m}.
\end{equation}
This is a useful observation.  For example
given the Taylor coefficients of one solution $P$, Equation \eqref{eq:rescaledCoefficients}
can be used to obtain a solution with more desirable decay rates (faster or slower decay).
}
\end{remark}

\subsection{Formalism and homological equations} \label{sec:homological}
We now return to Equation \eqref{eq:abstractODE}
under the assumptions given in the beginning of Section \ref{sec:aPrioriExist}.
Using the assumption that $(X, *)$ is a Banach 
algebra leads to an elegant formalism for Equation \eqref{eq:func_eq}. 
We begin by developing a few formulas.  

First consider
\begin{equation}\label{eq:powerSeriesAnstaz}
P(\theta) = \sum_{|m| = 0}^\infty p_m \theta^m,
\end{equation}
the power series of some analytic function $P \colon B_1 \to X$, 
and let $Q_2 \colon X \to X$ be the quadratic 
function defined by 
\[
Q_2(a) =  c * a * a,
\] 
where $c \in X$ is fixed.

A power series for the composition $Q_2 \circ P \colon B_1 \to X$ 
is given by 
\begin{eqnarray*}
Q_2(P(\theta)) &=& c* P(\theta) * P(\theta)  \\
&=& c* \left(   \sum_{|m| = 0}^\infty p_m \theta^m\right) *
\left(  \sum_{|m| = 0}^\infty p_m \theta^m \right) \\
&=& \sum_{|m| = 0}^\infty 
\sum_{m_1 + m_2 = m} c* p_{m_1} * p_{m_2} \theta^m \\
\end{eqnarray*}
Let $(Q_2 \circ P)_{m} \in X$ denote the power series
coefficients of the analytic function $Q_2 \circ P$. 
Matching like powers of $\theta$ leads to 
\[
(Q_2 \circ P)_m =  \sum_{m_1 + m_2 = m}  c* p_{m_1} * p_{m_2} = 2 c* p_0 * p_{m}  + 
\sum_{m_1 + m_2 = m}  \delta_{m_1, m_2}^{m} c* p_{m_1} * p_{m_2},
\]
where 
\[
\delta_{m_1, m_2}^{m} :=
\begin{cases}
0 & \mbox{if} \quad m_1 = m \quad \mbox{or} \quad m_2 = m   \\
1 & \mbox{otherwise}
\end{cases}.
\]
We write 
\[
p_m \diamond p_m :=
\sum_{m_1 + m_2 = m}  \delta_{m_1, m_2}^{m} c* p_{m_1} * p_{m_2},
\]
to denote the sum over terms with no $p_{m}$ dependance.
Noting that 
\[
2 c * p_0 * p_m = D Q_2(p_0) p_{m},
\]
we define the operator
\[
\tilde Q_2(P)_{m} : = c* p_m \diamond p_m,
\]
and have the formula
\begin{equation} \label{eq:quadraticAutoDiff}
(Q_2 \circ P)_{m} =  D Q_2(p_0) \, p_{m} + \tilde Q_2(P)_{m},
\end{equation}
where $ \tilde Q_2(P)_{m}$ depends on neither $p_m$ nor $p_0$.

More generally let $Q_n \colon X \to X$
be the monomial defined by 
\[
Q_n(a) = c* a^{\ast n},
\]
and consider the analytic function $Q_n \circ P \colon B_1 \to X$.
A nearly identical computation to the one given above shows
 that the $m$-th coefficient of the 
power series expansion of $Q_n \circ P$ is given by 
\begin{equation}\label{eq:monomialAutoDiff}
(Q_n \circ P)_{m} = D Q_n(p_0) p_m  + \tilde Q_n(P),
\end{equation}
where
\[
DQ_n(p_0) p_{m} = n p_0^{\ast n-1}  * p_m,
\] 
and we define 
\begin{equation}\label{eq:def_Qtilde}
\tilde Q_n(P)_{m} = \sum_{m_1 + \ldots + m_n = m} 
\delta_{m_1, \ldots, m_n}^{m} p_{m_1} * \ldots * p_{m_n},
\end{equation}
with
\[
\delta_{m_1, \ldots, m_n}^{m} = 
\begin{cases}
0 & \mbox{if } m_j = m \mbox{ for some } 1 \leq j \leq n \\
1 & \mbox{otherwise}
\end{cases}.
\]
We extend this formalism to the special case of $n=1$ 
by letting
$Q_1 \colon X \to X$ be 
\[
Q_1(a) := c*a,
\]
with $c \in X$ fixed.  The formula
\[
Q_1(P(\theta))_m = D Q_1(p_0) p_m + \tilde Q_1(P)_m,
\]
holds in this case as well once we note that 
\[
D Q_1(p_0) p_m = c * p_m,
\]
and define $\tilde{Q}_1(P)_m = 0$ for all $m \in \mathbb{N}^{d}$.

\bigskip

Assume in \eqref{eq:abstractODE} is given by a polynomial. This is also the case we study in the example described in \eqref{eq:gen_ODEsys} in the introduction.  
That is we can write the vector field as 
\[
F(x) = Ax + \sum_{n = 1}^s c_n * x^{*_n},
\]
where for $1 \leq n \leq s$ the  $c_n \in X$ are fixed.
We seek a solution of Equation \eqref{eq:func_eqBA} 
under these conditions.  

Since, under the hypotheses of Theorem \ref{thm:smt}, there exists an analytic 
solution of Equation \eqref{eq:func_eqBA}, we look for  
$P \colon B_1 \to X$ a satisfying the \textit{ansatz} of Equation \eqref{eq:powerSeriesAnstaz}. 
Imposing the first order constraints
given in Equation \eqref{eq:lin_BA} gives 
that the first order coefficients of the power series solution are given by 
\[
p_0 = \tilde x \quad \text{and} \quad p_{e_j} =  s_j \xi_j \quad \quad  \mbox{for } 1 \leq j \leq d,
\]
where $e_j$, $1 \leq j \leq d$ is the standard basis for the 
multi-indices $m \in \mathbb{N}^d$ with $|m| = 1$. Considering the left hand side of Equation \eqref{eq:func_eqBA} subject
to the power series ansatz leads to
\begin{eqnarray*}
F(P(\theta)) &=&  A P(\theta) + \sum_{n=1}^s c_n * (P(\theta))^{\ast n} \\
&=& AP(\theta) + \sum_{n=1}^s (Q_n \circ P)(\theta), 
\end{eqnarray*}
and after matching like powers of $\theta$ we see that 
the $m$-th power series coefficient of $F \circ P$ is 
given by 
\begin{equation}\label{eq:g_circ_P_alpha}
(F \circ P)_{m} = A p_{m} + \sum_{n=1}^s DQ_n(p_0) p_m
+ \sum_{n=1}^s (\tilde Q_n \circ P)_{m} = Dg(p_0)p_m + \sum_{n=1}^s  \tilde Q_n (P)_{m}.
\end{equation}

Similarly the right hand side of Equation \eqref{eq:func_eqBA} is, in light of 
the power series ansatz,  given by  
\begin{equation} \label{eq:DPLambdaTheta_alpha}
\lambda_1 \theta_1 \frac{\partial}{\partial \theta_1} P(\theta) + \ldots + 
\lambda_d \theta_d \frac{\partial}{\partial \theta_1} P(\theta)  =
\sum_{|m|=0}^\infty (m_1 \lambda_1 +\ldots + m_d \lambda_d) p_{m} \theta^{m}.
\end{equation}
Equating like powers of $\theta$ in 
Equation \eqref{eq:g_circ_P_alpha} and Equation \eqref{eq:DPLambdaTheta_alpha}
gives that for each $m \in \mathbb{N}^d$ with $|m| \geq 2$ the 
coefficient $p_m \in X$ is a solution of the equation 
\begin{equation} \label{eq:homologicalEqnBanachAlg}
[DG(p_0) - (m_1 \lambda_1 +\ldots + m_d \lambda_d) \mbox{Id}] p_m 
= -  \sum_{n=1}^s  \tilde Q_n (P)_{m},
\end{equation}
where the right hand side is in $X$ and depends on neither $p_m$ nor $p_0$.

Equation \eqref{eq:homologicalEqnBanachAlg} gives one equation for each $m$,
and we refer to these as the \textit{homological equations} for $P$.  We note that 
since $p_0 = \tilde x$, the linear operator $A_m$
defined by  
\[
A_{m} = DF(p_0) - (\lambda_1 +\ldots + \lambda_d) \mbox{Id},
\]
is a boundedly invertible linear operator on $X$ assuming that 
\[
m_1\lambda_1 +\ldots + m_d \lambda_d \notin \mbox{spec}(D F(\tilde x)).
\]
But $m_1, \ldots, m_d \geq 0$ and the $\lambda_1, \ldots, \lambda_d$
are positive real numbers, so that 
$m_1\lambda_1 +\ldots + m_d \lambda_d \in (0, \infty)$ for all $|m| \geq 2$.   
Since $\lambda_1, \ldots, \lambda_d$ are the only elements of the unstable spectrum 
of $DF(\tilde x)$ we have that $A_m$ is an isomorphism as long as  
\[
m_1\lambda_1 +\ldots + m_d \lambda_d \neq \lambda_j
\]
for $1 \leq j \leq d$, i.e. as long as the unstable eigenvalues are non-resonant in 
the sense of Definition \ref{def:nonRes}.  Then the homological equations
are uniquely and recursively solvable to all orders, and $P$ is formally 
well defined and unique in the sense of power series (once the (scaled) eigenvectors
$\xi_1, \ldots, \xi_d$ are fixed).

\subsection{Zero finding problem and the Newton-like operator for the unstable manifold}\label{sec:zero_fin}
In this section we interpret \eqref{eq:homologicalEqnBanachAlg} as a zero finding problem for the Taylor coefficients of the unstable manifold parameterization.
That is, we define a map $f$ such that $P$ given by \eqref{eq:powerSeriesAnstaz} solves \eqref{eq:func_eqBA} together with \eqref{eq:lin_BA} if and only of $f(p) = 0$, where $p$ is the sequence of coefficient sequences of $P$.

To begin consider
the set of all infinite sequences $p = \{ p_m \}_{m \in \mathbb{N}^d}$ 
with $p_m \in X$.  The sequence space is endowed with the norm
\[
\| p \|_1 := \sum_{|m| = 0}^\infty \|p_m\|,
\]
where the norm inside the sum is the $X$ norm.  Note that 
if $\{p_m\}_{m \in \mathbb{N}^d}$ are the Taylor coefficients of an 
analytic function on $\mathbb{B}_1$ then this is the $\mathcal{H}$
norm defined in Section \ref{sec:aPrioriExist}.  Let
\[
\ell^1_d(X) := \left\{ \{p_m\}_{m \in \mathbb{N}^d} \,,\hspace{3pt} p_m\in X: \, 
\| p \|_1 < \infty 
\right\},
\]
and note that this is a Banach space.

\paragraph{Zero finding map}
Returning to the discussion of formal series in Section 
\ref{sec:homological}, and in particular by considering again 
the formulas developed in Equations \eqref{eq:g_circ_P_alpha} 
and \eqref{eq:DPLambdaTheta_alpha}, we define the map
$b \colon \ell^1_d(X) \to \ell^1_d(X)$ component wise by  
\[
b_m(p) := Dg(p_0)p_m + \sum_{n=1}^s \tilde Q_n (p)_m,
\]
where $\tilde Q_n$ is as defined in Equation \eqref{eq:def_Qtilde}.
We now have the following.

\begin{definition}\label{def:zero_findingmap}
Let an equilibrium $\tilde{x}$ of \eqref{eq:abstractODE} together with eigenvalues 
$\lambda_i$ and eigenvectors $\xi_i$ for $i = 1,\ldots,d$ be given. 
Define the map on $\ell_d^1(X)$ by 
\begin{equation}\label{eq:f}
f_{m}(p) = \begin{cases}
p_{0}-\tilde{x} & m = 0\\
p_{e_i}-\xi_{i} & m = e_i\\
(\lambda\cdot m )p_{m}-b_{m}(p)& |m|\geq 2
\end{cases}
\end{equation} 
where $\lambda\cdot m \bydef \sum_{i = 1}^{d}\lambda_i m_i$.
\end{definition}
The following Lemma is an immediate consequence of the above definitions.

\begin{lemma}
Suppose that $p \in \ell_d^1(X)$ has $f(p) = 0$, where $f$ is given by \eqref{eq:f}.  
Then $P:\mathbb{B}_1\to \ell^{1}_{\nu}$ given by \eqref{eq:powerSeriesAnstaz} solves \eqref{eq:func_eqBA} together with \eqref{eq:lin_BA}.
\end{lemma}

\begin{remark}\label{rem:scaling_remark}{\em
The scalings of the eigenvectors are free in the definition of the map.  
In practice we choose the scalings as discussed in Section \ref{sec:appl}.
}
\end{remark}

\subsection{Fixed point operator in the Fourier-Taylor basis} \label{sec:ZeroFindingFT}
We now specify the operators $A$ and $A^{\dag}$ as promised in Section 
\ref{sec:nonlinearAnalysis}. Here the map $f$ corresponds to the invariance equation \eqref{eq:func_eq}, where $g$ is given by the infinite dimensional ODE system for the eigenbasis coefficients as specified in \eqref{eq:gen_ODEsys}. Therefore we  specialize to the 
case of a Fourier-Taylor base so that, combining the notation of Section  
\ref{sec:nands} and \ref{sec:zero_fin}, we let
 of $X = \ell_\nu^1$ so that $\ell_d^1(X) = X^\nu$.

We define the finite dimensional truncation of $f$, which we denote by $f^{MK}$. 
Let us set  for $M\in\mathbb{N}^d$ and $K>0$ the set 
\[
\mathcal{I}_{MK} = \left\{(m,k): m\preceq M\quad k\leq K\right\}
\] 
and the projection $\Pi_{MK}:X^{\nu}\to\mathbb{R}^{(|M|+1)(K+1)}$ by $\Pi_{MK}p = (p_{mk})_{(m,k)\in\mathcal{I}_{MK}}\bydef p^{MK}$. We identify $\mathbb{R}^{(|M|+1)(K+1)}$ as a subspace of $X^{\nu}$ by seeing an element  as a sequence of sequences, where each sequence entry $p_{mk}$ vanishes for either $k>K$ or $m\succ M$. We denote this operation formally by the immersion $\tau: \mathbb{R}^{(|M|+1)(K+1)} \to X^{\nu}$. Here $m\succ M$ means $m_i> M_i$ for at least one $i = 1,\ldots,d$. Then we assume a splitting of the map $f$ in the form 
\begin{equation}\label{eq:splitting_f}
f(p) = \tau(f^{MK}(p^{MK})) + f^{\infty}(p), 
\end{equation}
where $f^{MK}: \mathbb{R}^{(|M|+1)(K+1)}\to \mathbb{R}^{(|M|+1)(K+1)}$ is the map we implement numerically on the computer. In the following we will drop the immersion $\tau$ whenever it is simplifying the notation. Assume we compute an approximate zero $\bar{p}\in X^{\nu}$ of $f$, that is $f^{MK}(\bar{p}^{MK})\approx 0\in\mathbb{R}^{(|M|+1)(K+1)}$. To define the Newton-like fixed point operator we need an approximate inverse of $Df(\bar{p})$.
\begin{definition}
\begin{enumerate}
\item The following operator is an approximate inverse of $Df(\bar{p})$. Let $A^{MK}\approx Df^{MK}(\bar{p})^{-1}$ and set
\begin{equation}\label{eq:A}
(Ap)_{mk} = \begin{cases}
(A^{MK}p^{MK}))_{mk} & (m,k)\in\mathcal{I}_{MK}\\
p_{mk} & |m|\leq 1 \text{ and }k>K\\
\frac{1}{(\tilde{\lambda}\cdot m)-\mu_k} p_{mk} & (m,k)\notin \mathcal{I}_{MK}\text{ and }|m|> 1
\end{cases}
\end{equation}
\item If $A$ is injective, then fixed points of 
\begin{equation}\label{eq:T}
T: X^{\nu}\to X^{\nu},\quad Tp = p- A f(p) 
\end{equation}
correspond to zeros of $f$.
\end{enumerate}
\end{definition}
We also specify the operator $A^{\dag}\approx A^{-1}$:
\begin{equation}\label{eq:Adag}
(A^{\dag}p)_{mk} = \begin{cases}
(Df^{MK}(\bar{p})p^{MK})_{mk} & (m,k)\in\mathcal{I}_{MK}\\
p_{mk} & |m|\leq 1 \text{ and }k>K\\
(\tilde{\lambda}\cdot m - \mu_k)p_{mk} & (m,k)\notin \mathcal{I}_{MK}\text{ and }|m|> 1
\end{cases}.
\end{equation}


\section{Applications}\label{sec:appl}

Consider Fisher's equation with Neumann boundary conditions as specified in \eqref{eq:Fishers}. Because we impose Neumann boundary conditions we expand $u(x,t)$ in a Fourier cosine series and obtain the infinite system of ODEs for the real Fourier coefficients $a = (a_k)_{k\geq 0}$
\begin{equation}\label{eq:Fishers_seq}
a_{k}'(t) = (\alpha-k^2)a_{k}(t) + \displaystyle\sum_{\stackrel{k_1 + k_2 + k_3 = k}{k_i\in\mathbb{Z}}}c_{|k_1|}a_{|k_2|}a_{|k_3|}\bydef g_{k}(a)\end{equation}
where $c = (c_k)_{k\geq 0}$ is the sequence of real Fourier coefficients of $c(x)$.\\

\subsection{Validated computation of the first order data}\label{sec:firstorderdata}
In order to build a high order approximation of the unstable manifold of an equilibrium $\tilde{a}$ of \eqref{eq:Fishers_seq} we need a 
validated representation of $\tilde a$ and also its eigendata. In the following paragraphs we discuss how this 
is achieved. 
All of the computer programs discussed in this Section are available for download at \cite{parmPDEcode}.

\paragraph{Equilibrium solution:}
We look for an equilibrium  solution $\tilde{a}$ of Equation \eqref{eq:Fishers_seq}, that is we demand $g(\tilde{a}) = 0$.
The map is well defined as long as $1 < \bar \nu < \nu$, 
but unbounded for $\bar \nu = \nu$.  
In fact $g$ is Frechet differentiable with differential $Dg(a) \colon \ell^{1}_{\nu} \to \ell^1_{\bar{\nu}}$
given by 
\[
[Dg(a)h]_k =  (\alpha - k^2) h_k - 2 \alpha (c * a * h)_k, \quad \quad \quad k \geq 0,
\]
for $a, h \in \ell^{1}_{\nu}$.
Let us specify the operator $A$ and $A^{\dag}$ in this specific context. We choose $K >0$,
and define the projection $g^K \colon \mathbb{R}^{K+1} \to \mathbb{R}^{K+1}$ by 
\[
g^K(a^K) := (\alpha - k^2) a_k - \alpha (c^K * a^K * a^K)_k^K,
\]
where 
\[
(c^K * a^K * a^K)_k^K \bydef 
\sum_{\stackrel{k_1 + k_2 + k_3 = k}{-K \leq k_1, k_2, k_3 \leq K}} c_{|k_1|} a_{|k_2|} a_{|k_3|},
\]
is the truncated cubic discrete convolution. 

\begin{remark}
{\em
Note that even though the nonlinearity is only quadratic (as a function of $a \in \ell^1_\nu$) from a numerical 
point of view the nonlinearity requires computation of the full cubic discrete convolution.  For this 
we use the fast Fourier transform built into the IntLab library.  The reader interested in the 
details can find the map implemented in the program 
\begin{center}
\verb|fisherMapII_cos_intval.m|.
\end{center} 
Similarly the Jacobian matrix is computed using the FFT and standard shift operations.  Our 
implementation is in the program 
\begin{center}
\verb|fisherMapII_Differential_cos_intval.m|
\end{center}
}
\end{remark}

\bigskip

Now, if $\bar a^K$ is an approximate solution of $g^K = 0$ then we let $A^K$ be a 
numerical approximate inverse of the matrix $Dg^K(\bar a^K)$, i.e. suppose that 
$A^K$ is an invertible matrix with 
\[
\| \text{Id} - A^K Dg^K(\bar a^K) \| \ll 1.
\]
Define the linear operators $A$ and $A^\dagger$ by 
\begin{equation}\label{eq:defA_dag}
(A^\dagger h)_n = \begin{cases}
(Dg^K(\bar a^K) h^K)_k    & \text{if } 0 \leq k \leq K \\
(\alpha - k^2) h_k & \text{if } k \geq K+1
\end{cases}
\end{equation}
and 
\begin{equation} \label{eq:defA}
(A h)_k = \begin{cases}
(A^K h^K)_k    & \text{if } 0 \leq k \leq K \\
\frac{1}{\alpha - k^2} h_k & \text{if } k \geq K+1
\end{cases}.
\end{equation}

Let $\bar a \in \ell^{1}_{\nu}$ denote the inclusion of $\bar a^K$ into $\ell^{1}_{\nu}$. For the sake of completeness we include the following 
Lemma providing the $Y$- and $Z$- bounds fulfilling \eqref{eq:Y} and \eqref{eq:Z}. The proof 
is a computation similar to those in Section $5$ of \cite{AlJPJay}, and is discussed in detail in \cite{ck_analytic}.
The MatLab program
\begin{center}
\verb|fisherEquilibriumAnalyticProof.m|
\end{center}
implements the computations which check that the hypotheses of
Lemma \ref{thm:fisherEquilibRadPolyBounds} are satisfied.

\begin{lemma} \label{thm:fisherEquilibRadPolyBounds}
Suppose that $\sqrt{\alpha} < K+1$ and that  $c = c^K + c^\infty \in \ell^{1}_{\nu}$. Let 
\[
Y_0 := | A^K g^K (\bar a) |_\nu + \alpha | A^K |_{\ell^{1}_{\nu}} | \bar a |_\nu^2 | c^\infty |_\nu 
+ \alpha \frac{| c^\infty |_\nu | \bar a |_\nu^2}{(K+1)^2 - \alpha} + 
2 \sum_{k = K+1}^{3K} \alpha \frac{|(c^K * \bar a^K * \bar a^K)_k|}{k^2 - \alpha} \nu^k,
\]
\[
Z_0 := | \text{Id} - A^K Dg^K(\bar a^K) |_{\ell^{1}_{\nu}}, 
\]
\[
Z_1 := 2 \alpha \sum_{k=0}^K |A_{0k}^K| \beta_k + 4 \alpha \sum_{n=1}^K\left( \sum_{k=0}^K |A_{nk}^K| \beta_k\right) \nu^n
+ \frac{2 \alpha}{(K+1)^2 - \alpha} | c|_\nu | \bar a |_\nu,
\]
where 
\[
\beta_k := \max_{K+1 \leq j \leq 2K - k} \frac{| (c^K * \bar{a}^K)_{j+k} |}{2 \nu^j} 
+ \max_{K+1 \leq j \leq 2K + k} \frac{|(c^K * \bar a^K)_{j-k}|}{2 \nu^j},
\]
and 
\[
Z_2 = 2 \alpha  \max \left(   | A^K |_{B(\ell^{1}_{\nu})}, 1 \right)  \max \left( |c|_\nu , 1 \right).
\]
Then the constants 
\[
Y = Y_0, 
\]
and 
\[
Z(r) = Z_2 r - (1 - Z_0 - Z_1),
\]
satisfy \eqref{eq:Y} and \eqref{eq:Z}.  In particular, if $r > 0$ is a positive constant
with 
\[
Z(r) r + Y_0 < 0,
\]
then there is a unique $\tilde a \in B_r(\bar a) \subset \ell^{1}_{\nu}$ so that 
$g(\tilde a) = 0$.
\end{lemma}

\paragraph{Validated computation of eigenvalue/eigenvector pairs:}
Suppose now that $\tilde{u}$ is any equilibrium solution of Fisher's equation
with Neumann boundary conditions.
Linearizing about $\tilde{u}$ leads to the eigenvalue problem
\[
\frac{d^2 }{d x^2} \xi + \alpha \xi - 2 \alpha c \tilde{u} \xi = \lambda \xi, 
\quad \quad \quad \quad \xi'(0) = \xi'(\pi) = 0.
\]
Letting $\tilde{a}$ denote the sequence of cosine series coefficients for $\tilde{u}$ leads
to the Fourier space formulation
\[
(\alpha - k^2) \xi_k - 2 \alpha (c * \tilde{a} * \xi)_k = \lambda \xi_k, \,\,\,\,\,\,\,\,\,\,\,\, \mbox{for} \,\, k \geq 0, 
\]
for the cosine series coefficients of $\xi$. We note that this
is the precisely the eigenvalue problem
\[
Dg(\tilde{a}) \xi =  \lambda \xi,
\]
in the sequence space, in direct analogy with the case of a finite dimensional vector field.

As per the philosophy of the present work, we solve the eigenvalue problem via a zero finding argument.
Since a scalar multiple of an eigenvector is again an eigenvector 
it is necessary to append some scalar constraint in order to isolate a unique 
solution of the eigenvalue/eigenvector problem.  We choose $s \in \mathbb{R}$ and look for 
a solution $\xi \in \ell^{1}_{\nu}$ having $\xi_0 = s$. (The choice of phase condition is a convenience. 
Other phase conditions such as $| \xi |_{\nu} = 1$ or $\xi(0) = s$ would work as well, and can be
incorporated by making only minor modifications to the mappings defined below).

Define the mappings $h \colon \ell^{1}_{\nu} \to \mathbb{R}$ by 
\[
\tau(\xi) := \xi_0 - s,
\]
and $h \colon \mathbb{R} \times \ell^{1}_{\nu} \to \ell^{1}_{\nu}$ by 
\[
h(\lambda, \xi)_k := 
(\mu - k^2) \xi_k - 2 \mu (c* \tilde{a}  * \xi)_k - \lambda \xi_k,
\,\,\,\,\,\,\,\,\,\,\,\, k \geq 0.     
\]
We then define the mapping $H \colon \mathbb{R} \times \ell^{1}_{\nu} \to 
\mathbb{R} \times \ell^{1}_{\nu}$ by 
\[
H(\lambda, \xi) := 
\left(
\begin{array}{c}
  \tau(\xi)  \\
  h(\lambda, \xi)  \\   
\end{array}
\right)
\]
A zero of $H$ is an eigenvalue/eigenvector pair for the operator $Dg(\tilde{a})$.
In turn $\lambda$ is an eigenvalue for the PDE with eigenfunction given by 
\[
\xi(x) = \xi_0 + 2 \sum_{k=1}^\infty \xi_k \cos(kx).
\]
 
We note that the mapping 
$H$ is nonlinear due to the coupling term $\lambda \xi_k$
(i.e. we consider $\lambda$ and $\xi$ as simultaneous unknowns).

We will now construct a Newton-like operator in order to study the equation $H(\lambda,\xi) = 0$.
First we note that 
\[
D_\lambda \tau(\xi) = 0, \,\,\,\,\,\,\,\,\,\,  D_\xi h(\xi) = e_0,
\]
where $(e_0)_k = 1$ if $k = 0$ and is zero otherwise, that
\[
D_\xi h(\lambda, \xi)v 
 = \left[Dg(\tilde{a})  - \lambda \mbox{Id}\right] v,
\]
and that 
\[
D_\lambda H(\lambda, \xi) = -\xi.
\]
In block form we write
\[
DH(\lambda, \xi) (w, v) = 
\left(
\begin{array}{ccc}
  0 &  e_0  \\
 -\xi    &   Dg(\tilde{a})  - \lambda \mbox{Id}
\end{array}
\right)
\left[
\begin{array}{ccc}
  w   \\
  v   
\end{array}
\right],
\]
where $w \in \mathbb{R}$ and $v \in \ell_\nu^1$.
Consider the projected map
$h^K \colon \mathbb{R}^{K+2} \to \mathbb{R}^{K+1}$ defined by 
\begin{equation}
h^K(\lambda, \xi^K)_k :=
 (\alpha - k^2)\xi^K_k - 2 \alpha(c^K * \bar{a}^{K} * \xi^K )^K_k - \lambda \xi^K  
 \,\,\,\,\,\,\,\,\,\,\,\, 0 \leq  n \leq N  
\end{equation}
and define the total projection map $H^K: \mathbb{R}^{K+2} \to \mathbb{R}^{K+2}$
by 
\[
H^K(\lambda, \xi^K) = 
\left(
\begin{array}{ccc}
  \xi_0 - s   \\
  h^K(\lambda, \xi^K)  
\end{array}
\right).
\]
Suppose now that $(\bar \lambda, \tilde{\xi}^K)$ is an approximate solution of $G^K = 0$.

\begin{remark}{\em
In applications we choose a numerical eigenvalue/eigenvector pair
$(\bar \lambda, \bar{\xi}^K)$ for the matrix $Dg^K(\bar{a}^K)$.
We are free to solve the finite dimensional eigenvalue/eigenvector problem 
using any convenient linear algebra package. For an illustation of the type of validation result we obtain is given in Figure \ref{fig:eigenvalues}.
}

\begin{figure}[h!] 
\begin{center}
\includegraphics[width=8cm]{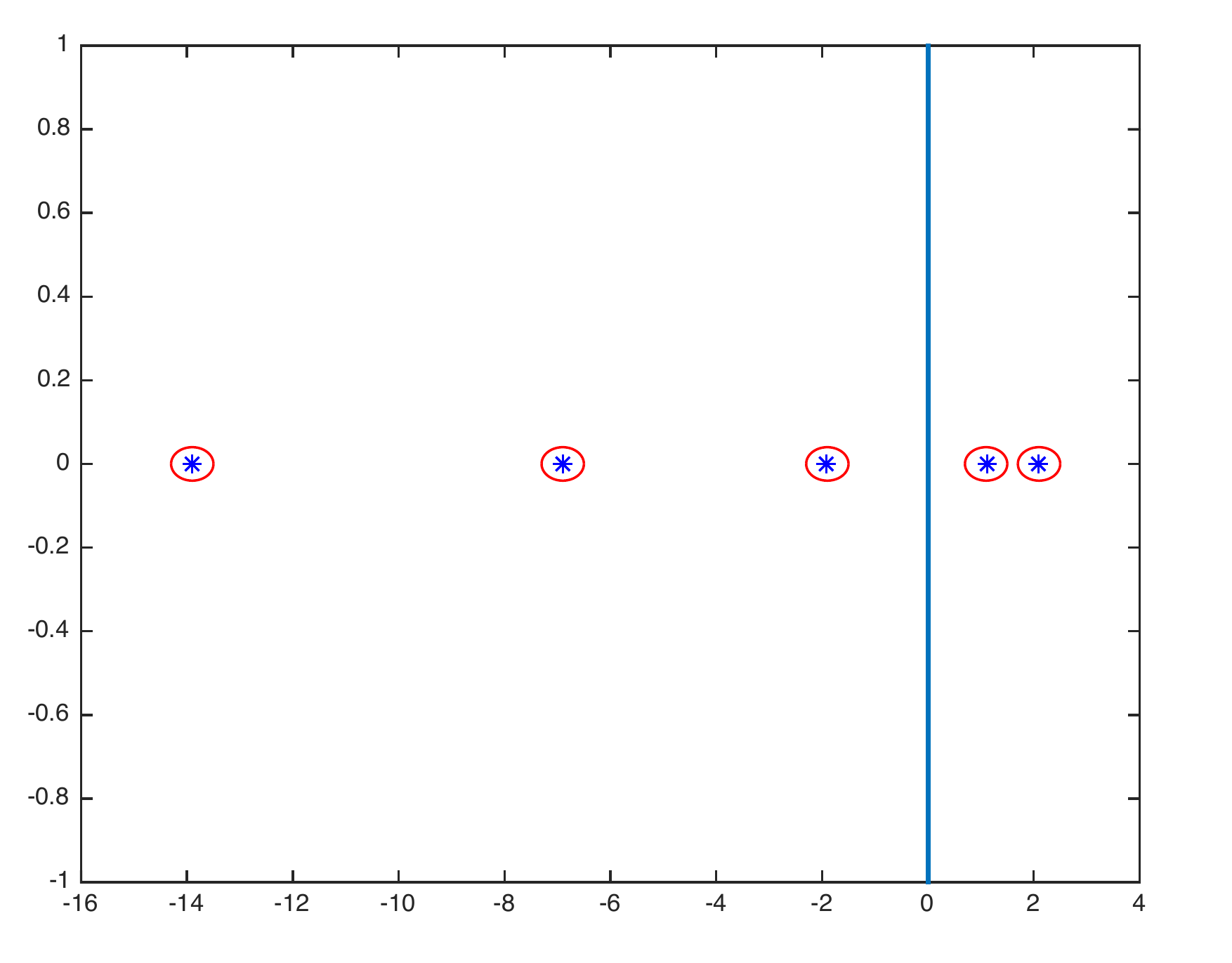}
\end{center}
\caption{
This is an illustration of the type of result our eigenvalue validation method yields. The blue crosses indicate the numerical eigenvalues $\bar{\lambda}$ of the finite dimensional matrix $Dg^{K}(\bar{a}^K)$. The red circles centered at the crosses indicate where true eigenvalues of $Dg(\tilde{a})$, that is the linearization of the \textit{infinite dimensional} map $g$ at the \textit{precise} equilibrium  $\tilde{a}$, can be found. It remains to be checked that the number of positive eigenvalues of $Dg(\tilde{a})$ is the same as the one of $A^{\dag}$ in \eqref{eq:defA_dag}, or $A$ in \eqref{eq:defA} equivalently. This is subject of Lemma \ref{lem:eigCountFisher}.
} \label{fig:eigenvalues}
\end{figure}

\end{remark}

\medskip

Let $B^K$ be a $K+2 \times K+2$ matrix which is obtained as a numerical 
inverse of $DH^K(\bar \lambda, \tilde{\xi}^K)$.
We partition $B^K$ as 
\[
B^K =
\left(
\begin{array}{cc}
   B_{11}^K    &  B_{12}^K   \\
    B_{21}^K   &  B_{22}^K        
\end{array}
\right),
\]
where $B_{11}^K \in \mathbb{R}$ is the first entry of $B^K$
, $B_{12}^K \in \left(\mathbb{R}^{K+1}\right)^*$ is the remainder of the first row of $B^K$,
$B_{21}^K \in \mathbb{R}^{K+1}$ is the remainder of the first column of $B^K$
and $B_{22}^K$ is the remaining $K+1 \times K+1$ 
matrix block.  The linear operators $B, B^\dagger$ are defined respectively by 
\[
B^\dagger :=
\left(
\begin{array}{cc}
   B_{11}^\dagger    &  B_{12}^\dagger   \\
    B_{21}^\dagger   &  B_{22}^\dagger        
\end{array}
\right),
\]
where the sub-operators are $B_{11}^\dagger \colon \mathbb{R} \to \mathbb{R}$
defined by 
\[
B_{11}^\dagger := 0,
\]
$B_{12}^\dagger \colon \ell^{1}_{\nu} \to \mathbb{R}$ defined by 
\[
B_{12}^\dagger (v)_k = v_k
\]
$B_{21}^\dagger \colon \mathbb{R} \to \ell^{1}_{\nu}$ defined by 
\[
B_{21}^\dagger(w) := 
\left\{
\begin{array}{ccc}
  -\tilde{\xi}_k w & & 0 \leq k \leq K \\
  0 & & k \geq K+1 
\end{array}
\right. 
\]
and $B_{22}^\dagger \colon Y^{\nu'} \to \ell^{1}_{\nu}$ defined by 
\[
B_{22}^\dagger(v)_k :=
\left\{
\begin{array}{ccc}
  \left[D h^K(\bar \lambda, \tilde \xi^K) v^K \right]_k & & 0 \leq k \leq K \\
  (\alpha - k^2) v_k & & k \geq K+1 
\end{array}
\right. ,
\]
and
\[
B :=
\left(
\begin{array}{cc}
   B_{11}    &  B_{12}   \\
    B_{21}   &  B_{22}        
\end{array}
\right),
\]
and $B_{11} \colon \mathbb{R} \to \mathbb{R}$
defined by 
\[
B_{11} := B_{11}^K,
\]
$B_{12} \colon \ell^{1}_{\nu} \to \mathbb{R}$ defined by 
\[
B_{12} (v) :=
\sum_{k=0}^K (B_{12}^K)_k v_k,
\]
$B_{21} \colon \mathbb{R} \to \ell^{1}_{\nu}$ defined by 
\[
B_{21} (w) := 
\left\{
\begin{array}{ccc}
  \left(B_{21}^K w\right)_k  & & 0 \leq k \leq K \\
  0 & & k \geq K+1 
\end{array}
\right. 
\]
and $B_{22} \colon \ell^{1}_{\nu} \to \ell^{1}_{\nu}$ defined by 
\[
B_{22} (v)_k :=
\left\{
\begin{array}{ccc}
  \left[ B_{22}^K v^K \right]_k & & 0 \leq k \leq K \\
  (\alpha - k^2)^{-1} v_k & & k \geq K+1 
\end{array}
\right. .
\]

Define the space 
\[
\mathcal{X}_\nu := \mathbb{R} \times \ell_\nu^1.  
\]
We write $x = (\lambda, \xi)$ for an element of $\mathcal{X}_\nu$.
We employ the product space norm on $\mathcal{X}_\nu$ so that 
\[
\| x\| = \| (\lambda, \xi)\| := \max \left( | \lambda|, | \xi |_\nu \right).
\]
Then we write 
\[
H(x) = H(\lambda, \xi),
\]
and for $y = (w, v) \in \mathcal{X}_\nu$ we have for example that 
\[
B \, y := 
\left(
\begin{array}{c}
  B_{11} w + B_{12} v  \\
  B_{21} w + B_{22} v  
\end{array}
\right),
\] 
(and similarly for $Dh(\tilde x) y$ and $B^\dagger y$).
Define the Newton-like operator
$\hat{T} \colon \mathcal{X}_\nu \to \mathcal{X}_\nu$
by 
\begin{equation}
T(x) = x - B H(x).
\end{equation}
The following lemma gives sufficient conditions that $T$ is a
contraction in a neighborhood of the approximate solution. 
The standard proof is a computation similar (for example) to that
carried out explicitly in Section $5$ of \cite{AlJPJay}.

\begin{lemma}\label{lem:stabilityCheck}
Suppose that $K+1 > \sqrt{\alpha}$, and 
for each $0 \leq k \leq K$ define the quantities 
\[
\hat{\alpha}_k := \sup_{K+1 \leq n \leq 2K - k} \frac{ \left| ( c^K * a^K )_{k+n} \right|}{2 \nu^n},
\]
and
\[
\hat{\beta}_k := \sup_{K+1 \leq n \leq 2K + k} \frac{ \left| ( c^K * a^K )_{n-k} \right|}{2 \nu^n}.
\]
Let $b_{ij}$ denote the entries of the $K+1 \times K+1$ matrix $B_{22}^K$.
Then the constants
\[
 \hat{Y}_0^1 :=
\left| B_{11}^K \right| \left|\tilde{\xi}_0^K - s\right| + 
\sum_{k=0}^K \left| (B_{12}^K)_k\right| \left|h^K(\bar \lambda, \tilde{\xi}^K)_k  \right| +
2 \alpha | B_{12}^K |_{(\ell_{\nu})^*} | \tilde \xi |_\nu 
\left(  | c^K |_\nu | a^\infty |_\nu + | c^\infty |_\nu | \tilde{a} |_\nu    \right),
\]
\begin{eqnarray*}
 \hat{Y}_0^2 &:=&
 \left| B_{21}^K \right|_\nu \left| \tilde{\xi}^K_0 - s \right|+
| B_{22}^K h^K(\bar \lambda, \tilde \xi) |_\nu + 
2 \alpha | B_{22}^N |_{B(\ell_\nu)} |\tilde \xi |_\nu
\left(  | c^K |_\nu | a^\infty |_\nu + | c^\infty |_\nu | \tilde{a}|_\nu    \right) \\
&+& 
2 \sum_{k = K+1}^{3K} 2\alpha  \frac{| \left( c^K * \bar{a}^{K} * \tilde{\xi}^K \right)_k |}{k^2 - \alpha} \nu^k +
2 \alpha \frac{| \tilde \xi |_\nu
\left(  | c^K |_\nu | a^\infty |_\nu + | c^\infty |_\nu | \bar{a} |_\nu    \right)}{(K+1)^2 - \alpha},
\end{eqnarray*}
\begin{eqnarray*}
\hat{Z}_1^2 &:=& 
2 \alpha \max_{0 \leq k \leq K} |b_{0k}|(| c^K |_\nu | a^\infty |_\nu + | c^\infty |_\nu | \bar{a} |_\nu) +
2\alpha \sum_{k=0}^K | b_{0k}| (\hat{\alpha}_k + \hat{\beta}_k) \\
&+& 2 \alpha \sum_{k=1}^K \left(
 \max_{0 \leq n \leq K}|b_{kn}| (| c^K |_\nu | a^\infty |_\nu + | c^\infty|_\nu |\bar{a}|_\nu)
+2 \sum_{n=0}^K  |b_{kn}| (\hat{\alpha}_k + \hat{\beta}_k)
\right) \nu^k \\
&+& \frac{2 \alpha}{(K+1)^2 - \alpha} \left( |c|_\nu | \tilde{a} |_\nu  + |\bar \lambda| \right),
\end{eqnarray*}
\[
\hat{Z}_0^1 := \left|\left(\mbox{Id}_{\mathbb{R}^{K+2}} - B^K D H^K(\tilde x^K)\right)_{11} \right|
+ \left|
\left(\mbox{Id}_{\mathbb{R}^{K+2}} - B^K DH(\tilde x^K)\right)_{12} \right|_{(\ell_\nu)^*},
\]
\[
\hat{Z}_0^2 := \left|\left(\mbox{Id}_{\mathbb{R}^{K+2}} - B^K D H^K(\tilde x^K)\right)_{21}\right|_\nu
+ \left|
\left(\mbox{Id}_{\mathbb{R}^{K+2}} - B^K D H^K(\tilde x^K)\right)_{22}
\right|_{B(\ell_\nu)},
\]
\[
\hat{Z}_1^1 := 0, 
\]
\[
\hat{Z}_2^1 := 0, 
\,\,\,\,\,\,\,\,\,\, \mbox{and} \,\,\,\,\,\,\,\,\,\,
\hat{Z}_2^2 := \|B\|,
\]
satisfy \eqref{eq:Y} and \eqref{eq:Z}, i.e. 
the polynomials 
\[
p_1(r) := Z_2^1 r^2 - (1 - Z_1^1 - Z_0^1) + Y_0^1, 
\]
and 
\[
p_2(r) := Z_2^2 r^2 - (1 - Z_1^2 - Z_0^2) + Y_0^2,
\]
are radii-polynomials for the eigenvalue/eigenvector problem. 
In particular, if $r$ is a 
positive constant having $p_1(r), p_2(r) > 0$ then there exists 
a unique pair $(\hat \xi, \hat \lambda)$ so that $\hat \xi \in B_r(\tilde \xi) \subset \ell^{1}_{\nu}$
and $|\hat \lambda - \bar \lambda| \leq r$ having that the pair 
solve the equation $\hat G = 0$, i.e. they are an eigenvalue/eigenvector
pair for Fisher's equation.
\end{lemma}

\paragraph{Correct eigenvalue count for the equilibrium:}
Now suppose that $\tilde{a} \in \ell_\nu^1$ is as in the previous sections, so that $g(\bar a) = 0$.  
Let $A \colon \ell_\nu^1 \to \ell_\nu^1$ be the linear operator defined by 
Equation \eqref{eq:defA}.  Moreover suppose that the $K+1 \times K+1$ matrix 
$A^K$ is diagonalizable, with eigenvalues
$\lambda_0, \ldots, \lambda_K \in \mathbb{C}$, and eigenvectors $\xi_0, \ldots, \xi_K \in \mathbb{C}^{K+1}$.
Letting $Q^K = [\xi_0, \ldots, \xi_K]$ and $\Sigma^K$ be the diagonal matrix 
of eigenvalues we have that 
\[
A^K = Q^K \Sigma^K Q^{-K},
\]
where $Q^{-K} := (Q^K)^{-1}$.  

Suppose that all of the eigenvalues have non-zero real part,
that exactly $m > 0$ are unstable, and that $\sqrt{\alpha} < K+1$.
Define the operators $Q, Q^{-1}, \Sigma \colon \ell_\nu^1 \to \ell_\nu^1$
by 
\[
(Qh)_k := \begin{cases}
[Q^K h^K]_k & 0 \leq k \leq K \\
h_k  & k \geq K+1
\end{cases},
\]
\[
(Q^{-1} h)_k := \begin{cases}
[Q^{-K} h^K]_k & 0 \leq k \leq K \\
h_k  & k \geq K+1
\end{cases},
\]
and 
\[
(\Sigma h)_k := \begin{cases}
[\Sigma^K h^K]_k & 0 \leq k \leq K \\
\frac{h_k}{\alpha - k^2}  & k \geq K+1
\end{cases}.
\]
Then note that 
\begin{itemize}
\item $\Sigma$ is well defined, 
\item $A$ and $\Sigma$ have the same spectrum,
\item the spectrum of $\Sigma$ and hence of $A$ is 
\[
\mbox{spec}(A) = \{ \lambda_0, \ldots, \lambda_K \} \cup \bigcup_{k = K+1}^\infty \frac{1}{\alpha - k^2} \cup \{0\},
\]
\item $\Sigma$ is a compact, 
\item $A = Q \Sigma Q^{-1}$.
\item The operator $B = Q \Sigma^{-1} Q^{-1}$ is unbounded but densely defined, due to the algebraic growth 
of the eigenvalues $\alpha - k^2$.  Since the eigenvalues approach $-\infty$, we have that $B$ 
generates a compact semi-group.  
\end{itemize}
The following lemma provides sufficient conditions that the matrix $A^K$ gives the correct 
unstable eigenvalue count for the infinite dimensional linearized problem.  
The MatLab program 
\begin{center}
\verb|fisher_validateEigCount.m|
\end{center}
performs the computations which check the hypotheses of Lemma \ref{lem:eigCountFisher}.

\begin{lemma} \label{lem:eigCountFisher}
Let $Y_0, Z_0, Z_1$ and $Z_2$ be the positive constants defined in Lemma
\ref{thm:fisherEquilibRadPolyBounds}, and suppose that $A, Q, Q$ and $\{\lambda_0, \ldots, \lambda_K\}$
are as discussed above.
Define 
\[
\mu_0 := \max_{0 \leq j \leq K} \sqrt{1 + \left(\frac{\mbox{{\em imag}}(\lambda_j)}{\mbox{{\em real}}(\lambda_j)} \right)^2},
\]
and suppose that $r > 0$ has that 
\[
Z_2 r^2 - (1-Z_0 - Z_1)r  + Y_0 < 0.
\]
Define 
\[
\epsilon := Z_2 r + Z_1 + Z_0, 
\]
and assume that 
\[
\| Q^K \| \| Q^{-K}\| \mu_0 \epsilon < 1.
\]
Then $Dg(\tilde a)$ has exactly $m$ unstable eigenvalues.
\end{lemma}

\begin{proof}
Note that 
\[
DT(x) = \mbox{I} - ADg(x),
\]
and inspection of Equation \eqref{eq:Z} implies that 
\[
\| \mbox{I} - ADg(x)\| \leq Z_2 r + Z_1 + Z_0 \leq \epsilon,
\]
for all $x \in B_r(\tilde x)$, in particular this inequality holds at $x = \tilde x$.
Then the operators $A$, $B = Q \Sigma^{-1} Q^{-1}$, 
and $M = B + H$ (with $H = Dg(\tilde x) - B$) satisfy the hypothesis of 
Lemma \ref{lem:eigCount} with 
\[
M = Dg(\tilde x),
\]
so that we have that correct eigenvalue count as claimed.
\end{proof}

\paragraph{Example numerical computation of the linear data with a non-constant spatial inhomogeneity:}\label{ex:equi_examp}
Consider Fisher's equation with $\alpha = 2.1$ and the spatial inhomogeneity 
given by a Poission kernel 
\begin{equation}\label{eq:poisson_kernel}
c(x) = 1 + 2 \sum_{k=1}^\infty r^k \cos(kx),
\end{equation}
with $r = 1/5$. As a matter of fact the Fourier coefficients $c_k = r^k$ are in $\ell^{1}_{\nu}$ whenever $r\nu<1$. We also have explicit control over the norm $|c|_{\nu} = \frac{2}{1-rv}-1$. To deal with the inhomogeneity $c(x)$ we split its sequence of Fourier coefficients into the form $c = \bar{c}+c^{\infty}$, where we have precise control over $|c^{\infty}|_\nu$. The Matlab script \verb|a_validateLinearData_paperVersion| carries out all the necessary computations to validate the equilibrium together with its stability data and is available at \cite{parmPDEcode}.

The numerically computed equilibrium solution $\bar{a}^2$
and its approximate unstable eigenfunction $\bar{\xi}$ are illustrated in Figure \ref{fig:fisherLinearData}. We refer to this solution as nontrivial and give it the number $2$ as there are two explicit equilibria given by the constant zero solution and the function $\frac{1}{c(x)}$.

\begin{figure}[h!] 
\begin{center}
\includegraphics[width=8cm]{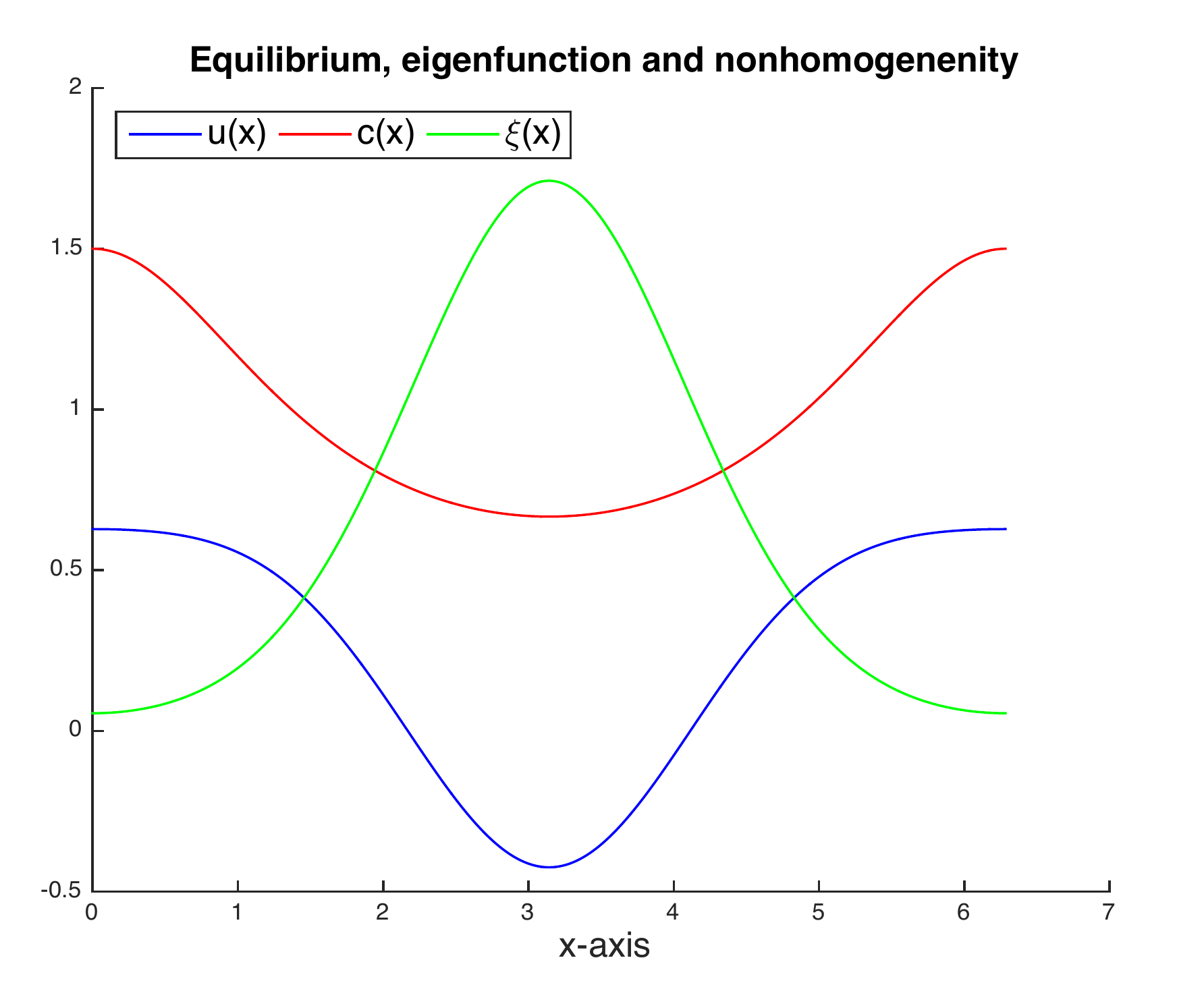}
\end{center}
\caption{
Linear data for the Fisher equation with $\alpha = 2.1$.
The red curve illustrates the spatial inhomogeneity with $c(x)$
a Poisson kernel with parameter $r = 1/5$.  The blue curve illustrates 
the numerically computed non-trivial equilibrium solution $\bar{a}^2$.  
The greed curve illustrates the numerically computed unstable 
eigenfunction $\bar{\xi}$.   The data is validated in $\ell^{1}_{\nu}$ with $\nu = 1.001$ and   
$C^0$ errors less than $5 \times 10^{-13}$.
} \label{fig:fisherLinearData}
\end{figure}

We approximate the system using $K = 20$ cosine modes, i.e. the numerical computations 
are carried out in $\mathbb{R}^{21}$.
We choose $\nu = 1.1$ and use the MatLab programs discussed in the preceding 
paragraphs to validate the results.  We obtain that there exists a true analytic equilibrium 
solution for the problem whose $C^0$ distance from the numerical approximation is 
less than $r_0 =  2.1 \times 10^{-14}$.  Similarly, we obtain that the equilibrium 
has exactly one unstable eigenvalue 
\[
\lambda_u = 2.194489888429804 \pm 3.5 \times 10^{-13},
\]
and obtain validated error bounds on the eigenfunction of the same order.

\subsection{Validated parameterization of the unstable manifold}\label{subsec:cmanifold}

First let us give a concrete formula for $b: X^{\nu,d}\to X^{\nu,d}$ as specified in \eqref{eq:f}:
\begin{equation}\label{eq:Fisher_dm}
b_{mk} = (-k^2+\alpha)p_{mk}-\alpha (c\ast (p^{\ast_{TF}2})_{m})_k , \quad |m|\geq 0, k\geq 0,
\end{equation}
where $d = 1$ or $d = 2$ in the following. We start by considering the one-dimensional unstable manifold at the nontrivial equilibrium considered in the example of the last paragraph of \ref{ex:equi_examp}. In this case we highlight how our analysis works if the fixed point together with the eigenvalues and eigenvectors are only known as numerical values together with bounds on the truncation error obtained by methods described in Section \ref{sec:firstorderdata}. This in particular includes checking the Morse index. Then we describe the computation of a two-dimensional manifold at the origin in order to show that our analysis carries over to higher dimensional manifolds.\\

In the following we choose the parameter $\alpha = 2.1$.  The parameter $\alpha$ is an eigenvalue parameter for the zero solution, in the sense that for $(l-1)^2<\alpha<l^2$ the linearization $Dg(\tilde{a}^{0})$ has exactly l unstable eigenvalues. Hence by fixing $\alpha = 2.1$ we obtain a 2D unstable manifold at the origin.

\paragraph{1D unstable manifold at a nontrivial equilibrium}

From Section \ref{sec:firstorderdata} we are given the exact linear data $\tilde{a}$, $\tilde{\lambda}$ and $\tilde{\xi}$ in the form:

\begin{subequations}\label{eq:linear_data}
\label{eq:a0a1}
\begin{alignat}{1}
\tilde{a}^2 &= \bar{a}^2 + a^{\infty}\quad \text{with}\quad |a^{\infty}|_{\nu}\leq r_{\bar{a}^2} \label{eq:equilibrium}\\
\tilde{\xi} &= \bar{\xi} + \xi^{\infty}\quad \text{with}\quad |\xi^{\infty} |_{\nu}<r_{\tilde{\xi}}\quad \text{for a given} \quad\nu>1 \label{eq:eigenvectors}\\
\tilde{\lambda} &= \bar{\lambda} + \lambda^{\infty},\quad \text{with} \quad|\lambda^{\infty}|<r_{\tilde{\lambda}}\label{eq:eigenvalues},
\end{alignat}
\end{subequations}

where we recall that $Dg(\tilde{a})\tilde{\xi} = \tilde{\lambda}\tilde{\xi}$. In addition by checking the conditions in Lemma \eqref{lem:eigCountFisher} we ensure that the $\tilde{\lambda}$ indeed is the only positive eigenvalue of $Dg(\tilde{a})$. \\

In order to derive the bounds from Definition \ref{def:YZp} we first need to make precise how we split the map $f$ from \ref{def:zero_findingmap} into the form \eqref{eq:splitting_f}.

\begin{definition}\label{def:fsplitting1D}
Assume we are given truncation dimensions $M>0$ and $K >0$. Split $p = \Pi_{MK}p + p^{\infty}$, where $p^{\infty} = \left(Id-\Pi_{MK}\right)p$. We set 
\begin{equation}
f^{MK}(p^{MK})_{mk} = \begin{cases}
\bar{a}_k-(p^{MK})_{0k} & m = 0 \\
\bar{\xi}_k - (p^{MK})_{1k} & m = 1\\
(\bar{\lambda}\cdot m + k^2-\alpha) p^{MK}_{mk} + \alpha (\bar{c}\ast\left(p^{MK}\ast_{TF}p^{MK}\right)_{m}^{MK})_k^{K} & 2\leq m\leq M, \\&0\leq k \leq K
\end{cases}
\end{equation}
and 
\begin{equation}\label{eq:f_infty}
(f^{\infty}(p))_{m} = \begin{cases}
a^{\infty}_k - p^{\infty}_{0k} & m = 0, k\geq 0\\
\xi^{\infty}_k - p^{\infty}_{1k}& m = 1, k\geq 0\\
(\lambda^{\infty}\cdot m ) (\Pi_{MK}p)_{mk} + (\tilde{\lambda}\cdot m+k^{2}-\alpha)p^{\infty}_{mk}+\\
\alpha (Id-\Pi_{MK})\left(\bar{c}\ast(p^{MK}\ast_{TF}p^{MK}\right)_{m})_k+\\
\alpha\left[ c\ast((2p^{MK}\ast_{TF}p^{\infty})_{mk} + (p^{\infty}\ast_{TF}p^{\infty})_{m})_k\right]& m\geq 2 , k\geq 0
\end{cases}.
\end{equation}
Then $f(p) = f^{MK}(p^{MK}) + f^{\infty}(p)$.
\end{definition}

The following Theorems summarize the $Y$-bounds and $Z$-bounds from Definition \ref{def:YZp}. Note that Theorem \ref{th:Y1D_c1} rigorously controls the numerical residual and Theorem \ref{th:Z1D_c1} controls the contraction rate. We will split the derivative in the following way:
\begin{equation}
DT(\bar{p}+ru)rv = (Id-AA^{\dag})rv - A\left[(A^{\dag}-Df(\bar{p}+ru))rv\right].
\end{equation}
Note that we (for convenience) use a refined definition for $A^{\dag}$, where we replace $\mu_k = -k^2$ from \eqref{eq:Adag} by $\mu_k - \alpha$. The motivation for this splitting is that the first term is expected to be small and the second one is convenient to keep under explicit control. To bound the norm we use 
\[
\|DT(\bar{p}+ru)rv\|_{\nu}\leq \|(Id-AA^{\dag})rv\|_{\nu} + \|A\left[(A^{\dag}-Df(\bar{p}+ru))rv\right]\|_{\nu}.
\]
To structure later estimates let us define $\Delta\in X^{\nu,1}$
\begin{equation}
\Delta_{mk}(u,v) = \left[(A^{\dag}-Df(\bar{p}+ru))rv\right]_{mk},
\end{equation}
which will be of the form
\begin{equation}\label{eq:Delta}
\Delta = r\Delta^{(1)} + r^2 \Delta^{(2)}.
\end{equation}
 
\begin{theorem}\label{th:Y1D_c1} Y-bounds - 1D\\
Assume truncation dimensions $M>0$ and $K>0$ and an approximate zero $\bar{p}$, with $\Pi_{MK}\bar{p} = \bar{p}$ to be given. Define  
\begin{equation}\label{eq:Y_MK}
Y_m^{MK} = 
|(Df^{MK}(\bar{p})\bar{p})_{m0}| + 2\displaystyle\sum_{k = 1}^{K}|(Df^{MK}(\bar{p})\bar{p})_{mk}|\nu^{k}, \quad m = 0,\ldots, M, 
\end{equation}
and 
\begin{equation}\label{eq:Y_infty}
Y_{m}^{\infty} = \begin{cases}
r_{\bar{a}} & m = 0\\
r_{\bar{\xi}} & m = 1\\
|(|Df^{MK}(\bar{p})|\delta^{\infty})_{m0}| + 2\displaystyle\sum_{k = 1}^{K}|(|Df^{MK}(\bar{p})|\delta^{\infty})_{mk}|\nu^{k} +\\
\displaystyle2\sum_{k = K+1}^{3K}\frac{\alpha|(\bar{c}\ast(\bar{p}\ast_{TF}\bar{p})_m)_{k}|}{\tilde{\lambda}\cdot m+k^{2}-\alpha}\nu^{k} & m = 2,\ldots M\\
\frac{\alpha |(\bar{c}\ast\bar{p}\ast_{TF}\bar{p})_{m0}|}{\tilde{\lambda}\cdot m} + 2 \displaystyle\sum_{k = 1}^{3K}\frac{\alpha|(\bar{c}\ast(\bar{p}\ast_{TF}\bar{p})_m)_{k}|}{\tilde{\lambda}\cdot m+k^{2}-\alpha}\nu^k & M+1 \leq m\leq 2M,
\end{cases}
\end{equation}
where $\delta^{\infty}$ with $\Pi_{MK}\delta^{\infty} = \delta^{\infty}$ is given by 
\begin{equation}\label{eq:delta_infty}
\delta^{\infty}_{mk} = \begin{cases}
|\bar{p}_{mk}|\lambda^{\infty}  & 2\leq m\leq M, 0\leq k\leq K.
\end{cases}
\end{equation}
Then 
\begin{equation}\label{eq:YboundFisher1D} 
Y = \displaystyle\sum_{m = 0}^{M}Y^{MK}_m +\displaystyle\sum_{m = 0}^{M}Y^{\infty}_m + \displaystyle\sum_{m = M+1}^{2M}Y^{\infty}_m
\end{equation}  fulfills \eqref{eq:Y}.
\end{theorem}
\begin{proof}
To derive \eqref{eq:Y_MK} and \eqref{eq:Y_infty} we notice that $p^{\infty} = 0$ for $p = \bar{p}$. Hence formula \eqref{eq:f_infty} reduces to 
\begin{equation}\label{eq:f_infty}
(f^{\infty}(\bar{p}))_{m} = \begin{cases}
a^{\infty}_k & m = 0, k\geq 0\\
\xi^{\infty}_k& m = 1, k\geq 0\\
(\lambda^{\infty}\cdot m ) (\Pi_{MK}\bar{p})_{mk} + \\
\alpha (Id-\Pi_{MK})\left(\bar{c}\ast(\bar{p}^{MK}\ast_{TF}\bar{p}^{MK}\right)_{m})_k
& m\geq 2 , k\geq 0
\end{cases}.
\end{equation}
We see that $\delta^{\infty}_{mk}$ in \eqref{eq:delta_infty} is a component-wise bound for the terms in $|(f^{\infty}(\bar{p}))_{mk}|$ with $m\geq 2$ known to use only via error bounds. Next recall that we demand $Y$  to fulfill $\|Af(\bar{p})\|_{\nu}\leq Y$. Note that even though $\bar{p}$ only has finitely many non-zero components, $f(\bar{p})$ does not, due to the linear data being infinite dimensional. Now $Y^{MK}_m + Y_{m}^{\infty}$ bounds $|(Af(\bar{p}))_m|_\nu$ by definition of $A$ in \eqref{eq:A} for $m = 0, \ldots, M$ and $Y_{m}^{\infty}$ for $m = M+1,\ldots, 2M$. Then \eqref{eq:YboundFisher1D} follows from the definition of $\|\cdot\|_{\nu}$.
\end{proof}
As a preparation for the $Z$-bounds we need the following Lemma, see also \cite{kotParm}. We state the more general version with $l$ being a multi-index as this will be needed later on for higher dimensional manifolds. 

\begin{lemma}\label{lem:boundybound1D}
Let $v\in \mathbb{B}_{1}(0)\subset \ell^1_\nu$, truncation dimensions $M\in\mathbb{N}^d$ and $K>0$ and $\bar{p} = \Pi_{MK}\bar{p}\in X^{\nu}$ be given. Then for every $l\preceq M$ and for each $0\leq k\leq K$ the following  estimate is valid:
\begin{equation}
\left|(\bar{c}\ast\bar{p}_{l}\ast v^{\infty})_{k} \right|\leq  h_{lk}^1(\bar{p}) + h_{lk}^2(\bar{p}) ,
\end{equation}
with 
\[
 \max_{j = K+1,\ldots,2K-k}\frac{|(\bar{c}\ast\bar{p}_{l})_{j+k}|}{2\nu^{j}}\bydef h_{lk}^1(\bar{p})\quad\text{and}\quad \max_{j = K+1,\ldots,2K+k}\frac{|(\bar{c}\ast\bar{p}_{l})_{j+k}|}{2\nu^{j}}\bydef h_{lk}^2(\bar{p}).
\]
We set the convention $h(l,0) = 0$ and 
\begin{equation}
v^{\infty}_{k} = \begin{cases}
0 & 0\leq k\leq K\\
v_{k} & k \geq K+1
\end{cases}.
\end{equation} 
\end{lemma}
\begin{proof}
See \cite{kotParm}.
\end{proof}
\begin{theorem}\label{th:Z1D_c1} Z-bounds - 1D\\
Assume truncation dimensions $M>0$ and $K>0$ and an approximate zero $\bar{p}$, with $\Pi_{MK}\bar{p} = \bar{p}$ to be given. Define 
\begin{subequations}\label{eq:linear_data}
\label{eq:a0a1}
\begin{alignat}{1}
|\Delta^{(1)}|_{mk} &= \begin{cases}
0 & m = 0,1, k\geq 0\\
\frac{r_{\lambda}m}{\nu^{k}} + \displaystyle\sum_{l = 0}^{m}h_{lk}(\bar{p}) +m\|p\|_\nu|c^{\infty}|_{\nu}& 2\leq m\leq M, 0\leq k\leq K\\
0 & (m\geq M+1,k\geq 0) \vee (0\leq m\leq M, k\geq K+1)
\end{cases}\\
|\Delta^{(2)}|_{mk}&=\begin{cases}
0 & m = 0,1, k\geq 0\\
\frac{2\alpha|c|_{\nu}}{\nu^{k}} & 2\leq m\leq M, 0\leq k\leq K\\
0 & (m\geq M+1,k\geq 0) \vee (0\leq m\leq M, k\geq K+1)
\end{cases}\end{alignat}
\end{subequations},
\begin{equation}
\Sigma^{(j)}_{MK} = \displaystyle\sum_{m = 0}^{M}\left(|(|A||\Delta^{(j)}|)_0| + \displaystyle\sum_{k = 1}^{K}|(|A||\Delta^{(j)}|)_k|
\nu^k\right).
\end{equation}
and $\epsilon$ such that $\sup_{v\in B_{1}(0)}\|(Id-AA^{\dag})rv\|_{\nu}\leq \epsilon r$. Then 
\begin{equation}\label{eq:Z_boundFisher1D}
\begin{aligned}
Z(r) = &r\left(\epsilon + \Sigma_{MK}^{(1)}+ 2\alpha|c|_\nu\|\bar{p}\|_{\nu}\left(\frac{1}{K^{2}-\alpha}+\frac{1}{|(\bar{\lambda}+r_{\lambda})M-\alpha|}\right)\right)\\
&+r^{2}\left(\Sigma_{MK}^{(2)}+ 2\alpha|c|_\nu\left(\frac{1}{K^{2}-\alpha}+\frac{1}{|(\bar{\lambda}+r_{\lambda})M-\alpha|}\right)\right)
\end{aligned}
\end{equation}
fulfills \eqref{eq:Z}.
\end{theorem}

\begin{proof}
We start by expanding the difference $\Delta = Df(\bar{p}+ru)rv-A^{\dag}rv$ from equation \eqref{eq:Delta}.  Using the formula 
\[
(Df(p)q)_{mk} = \begin{cases}
q_0 & m = 0, k\geq 0\\
q_1 & m = 1, k\geq 0\\
(\tilde{\lambda}m + k^2-\alpha)q_{mk} + 2\alpha(c\ast (p\ast_{TF} q)_m)_k & m\geq 2 , k\geq 0
\end{cases} ,
\]
we obtain 
\[
Df(\bar{p}+ru)rv = \begin{cases}
rv_0& m = 0, k\geq 0\\
rv_1& m = 1, k\geq 0\\
(\tilde{\lambda}m + k^2-\alpha)rv_{mk} + 2\alpha r(c\ast (\bar{p}\ast_{TF} v)_m)_k& m\geq 2 , k\geq 0\\
2\alpha r^2(c\ast (u\ast_{TF} v)_m)_k
\end{cases}.
\]
Using the refined definition  of $A^{\dag}$ alluded to above (all linear terms cancel) we obtain for $\Delta$
\[
\Delta_{mk} = \begin{cases}
0 & m = 0,1\\
(\lambda^{\infty}m)p_{mk} + 2\alpha r(\bar{c}\ast (\bar{p}\ast_{TF} v^{\infty})_m)_k & 2\leq m\leq M, 0\leq k\leq K\\
2\alpha r(c^{\infty}\ast (\bar{p}\ast_{TF} v)_m)_k + 2\alpha r^2(c^{\infty}\ast (u\ast_{TF} v)_m)_k \\
2\alpha r(c\ast (\bar{p}\ast_{TF} v)_m)_k + 2\alpha r^2(c^{\infty}\ast (u\ast_{TF} v)_m)_k & m>M \vee k>K,
\end{cases}
\]
where we use the notation $v^{\infty}$ from Lemma \ref{lem:boundybound1D}. Now using this very Lemma and the decay information on $u,v$ we obtain \eqref{eq:a0a1}. The proof follows from applying $A$ to $\Delta$, taking the norm $\|\cdot\|_{\nu}$ and using 
\[
\frac{1}{|\tilde{\lambda}| m + k^2 -\alpha} \leq\begin{cases}
\frac{1}{|\tilde{\lambda}(M+1)| -\alpha} & m>M, k\geq 0\\
\frac{1}{(K+1)^2 -\alpha} & m\geq 0, k>K
\end{cases}.
\] 
\end{proof}
This provides us with the ingredients to define the radii polynomials by 
\[
p(r) = Y+Z(r)-r,
\]
with $Y$ given by \eqref{eq:YboundFisher1D} and $Z(r)$ given by \eqref{eq:Z_boundFisher1D}. The Matlab script
\begin{center}
 \verb|validateparametrization_c_intval.m| 
 \end{center} 
 implements these bounds. 
\begin{remark}
Using a radius $r_P$ such that $Y+Z(r_p)-r_p<0$ we can now rigorously enclose the image $P(\theta)\in\ell^{1}_\nu$ of every $\theta\in B_{1}(0)$. More precisely if $\|\tilde{p}-\bar{p}\|_{\nu}<r_P$ then $|P(\theta)-P^{MK}(\theta)|_{\nu}<r_P$. This allows is to bound the distance of $P(\theta)$ to any point using the triangle inequality. This will be used in our connecting orbit proof in Section \ref{sec:connorbit}.
\end{remark}

\paragraph{Results}
We take $\nu = 1.1$ and choose truncation dimensions in Fourier direction to be $K = 20$ and in Taylor direction $M = 60$. In Figure \ref{fig:decayplots} we illustrate the decay in these two directions. In Figure \ref{fig:flowbackward} we show an illustration of the manifold in function space and demonstrate the conjugation property of the parametrization map $P$. The corresponding computations are to befound in \verb|unstable1D_c_intval.m|

\begin{figure}[h!] \label{fig:decayplots}
\begin{center}
\includegraphics[width=8cm]{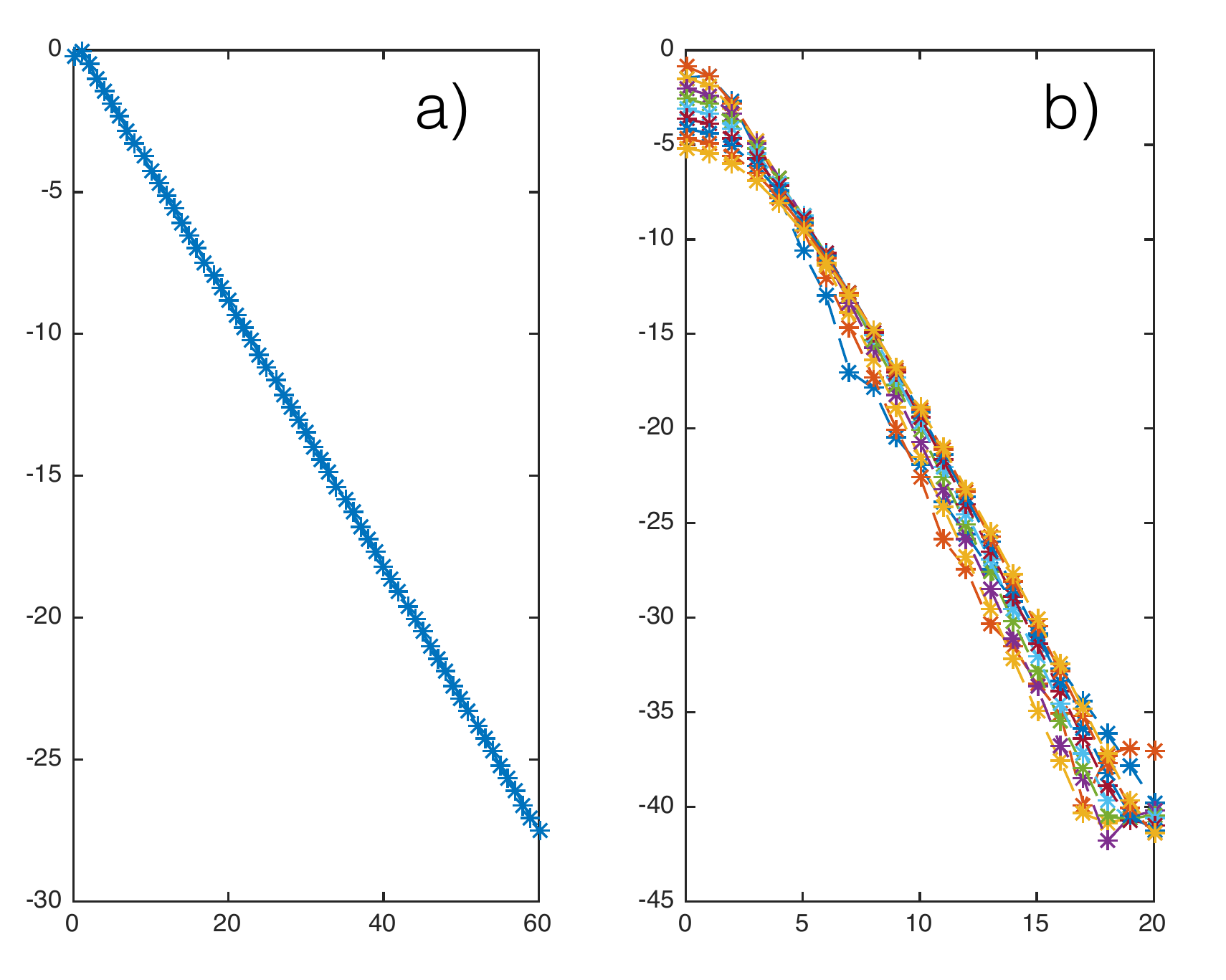}
\end{center}
\caption{a) We show the decay of $|p_m|_\nu$ with $m$. b) This illustrates the decay of the individual coefficient sequences $p_m\in\ell^1_\nu$ for $m = 0,\ldots,60$. Our approach quantifies the truncation error in both directions. } \label{fig:fisherLinearData}
\end{figure}

\begin{figure}[h!] \label{fig:flowbackward}
\begin{center}
\includegraphics[width=8cm]{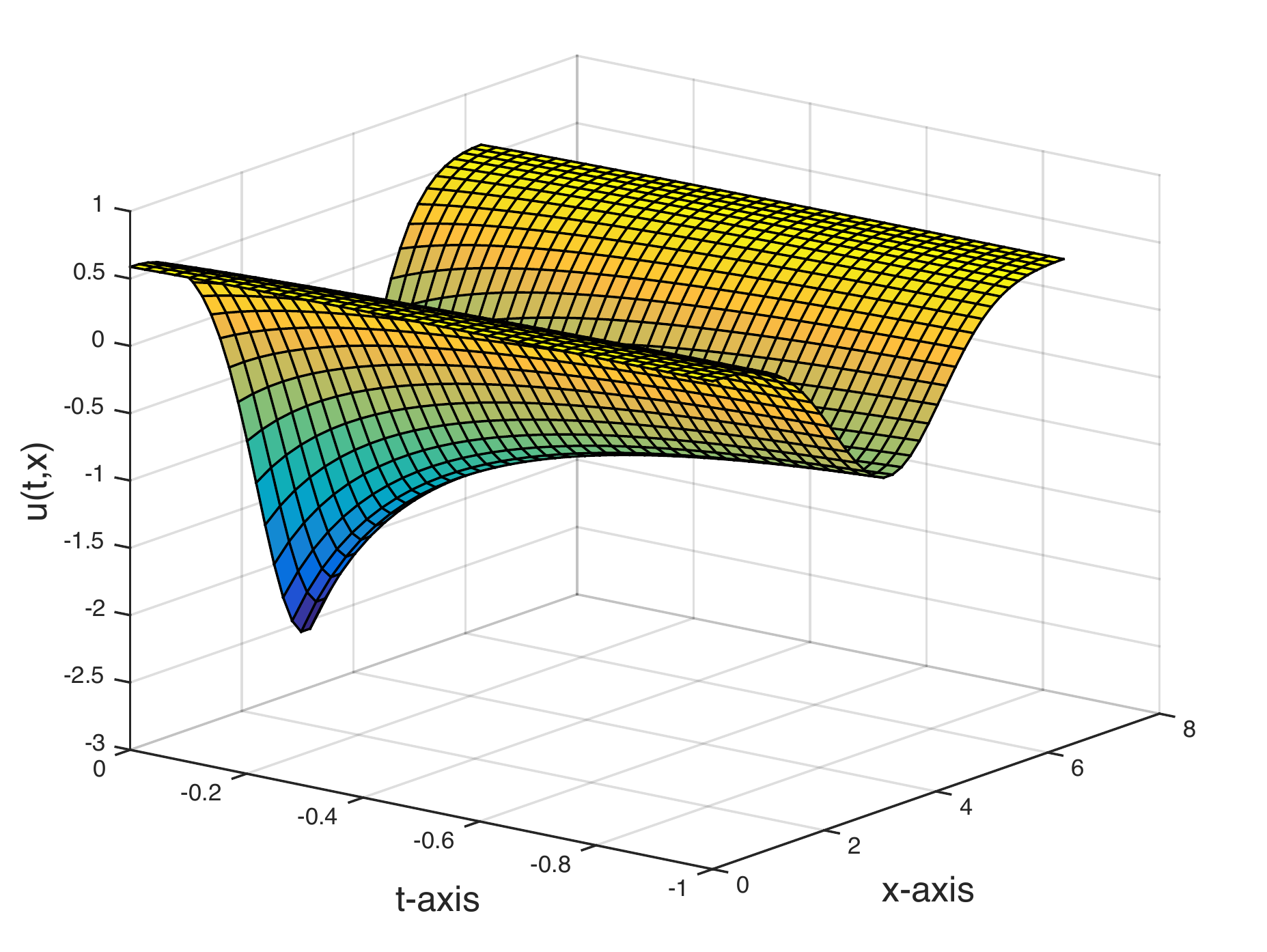}
\end{center}
\caption{3D image of the unstable manifold in function space. We use the fact that $a(t) = P(\exp(\tilde{\lambda}t))$ solves \eqref{eq:Fishers_seq} backward in time. We integrate the initial condition $P(0.9)$ backward for one time unit. Observe the convergence towards $\tilde{a}^2$.} \label{fig:fisherLinearData}
\end{figure}

\paragraph{Two unstable eigenvalues (at the origin):} 
 As alluded to above the Morse index of the origin is 2 which is clearly seen from the fact that for $b\in \ell^{1}_{\nu}$ we have that $(Dg(0)b)_k = (\alpha-k^2)b_k$ for $k\geq 0$. This implies that the unstable eigenvalues are $\tilde{\lambda}_1 = \alpha$ and $\tilde{\lambda}_2 = \alpha-1>0$ together with the eigenvectors $\tilde{\xi}_1 = (1,0,0\ldots)$ and $\tilde{\xi}_2 = (0,1,0,\ldots)$. Note that we still use tildas even though the quantities are not computed by computer-assisted means. Assume an approximate solution $\bar{p}$ corresponding to truncation dimensions $M\in\mathbb{N}^2$ with $M_i> 1$ $(i = 1,2)$ and $K>0$ to be given. In Figure \ref{fig:2Dunstable} we show an illustration of a 2D unstable manifold with truncation dimensions $M = (5,20)$ and $K = 20$. \\
 
 To explain how the estimates can be adapted to this setting let us again specify the splitting of the according to \eqref{eq:splitting_f}. Note that we deal with two-dimensional multi-indices for $m$ now but have exact linear data $\tilde{\lambda}, \tilde{\xi}$ for the equlibrium $\tilde{a}$ at the origin  at our disposal.
 
 \begin{definition}\label{def:fsplitting2D}
Assume we are given truncation dimensions $M\in\mathbb{N}^2$ and $K >0$. Split $p = \Pi_{MK}p + p^{\infty}$, where $p^{\infty} = \left(Id-\Pi_{MK}\right)p$. We set 
\begin{equation}
f^{MK}(p^{MK})_{mk} = \begin{cases}
-(p^{MK})_{0k} & m = 0 \\
(\tilde{\xi_i})_k - (p^{MK})_{e_ik} & m = e_i \quad (i = 1,2)\\
(\tilde{\lambda}\cdot m+ k^2-\alpha) p^{MK}_{mk} + \alpha (\left(p^{MK}\ast_{TF}p^{MK}\right)_{m}^{MK})_k^{K} & |m|\geq 2, m\preceq M \\&0\leq k \leq K
\end{cases}
\end{equation}
and 
\begin{equation}\label{eq:f_infty}
(f^{\infty}(p))_{m} = \begin{cases}
 - p^{\infty}_{0k} & m = 0, k\geq 0\\
- p^{\infty}_{1k}& m =  e_i \quad (i = 1,2), k\geq 0\\
(\tilde{\lambda}\cdot m+k^{2}-\alpha)p^{\infty}_{mk}+\\
\alpha (Id-\Pi_{MK})\left((p^{MK}\ast_{TF}p^{MK}\right)_{m})_k+\\
\alpha\left[((2p^{MK}\ast_{TF}p^{\infty})_{mk} + (p^{\infty}\ast_{TF}p^{\infty})_{m})_k\right]& |m|\geq 2 , k\geq 0
\end{cases}.
\end{equation}
Recall $\tilde{\lambda}\cdot m = \tilde{\lambda}_1m_1 + \tilde{\lambda}_2m_2 $. Note that in this case we can explicitly check the non-resonance condition from Definition \ref{def:nonRes} by checking that $\frac{\tilde{\lambda}_i}{\tilde{\lambda}_j}\notin \mathbb{N}$ for $i,j = 1,2$. Then $f(p) = f^{MK}(p^{MK}) + f^{\infty}(p)$.
\end{definition}

Following the same steps as we did in the previous Section we can derive $Y$- and $Z$- bounds. We can formulate the corresponding Theorems to \ref{th:Y1D_c1} and \ref{th:Z1D_c1}. For the implementation we refer to \verb|validateparametrizationorigin2D.m| to be found at \cite{parmPDEcode}. For a specific example calculation see \verb|manifoldorigin2D.m|, where we compute and validate the unstable manifold for the specific truncation dimension $M = [5,20]$. We choose  $\nu = 1.01$ and obtain a validation radius of $r = 5.978461\times 10^{-10}$. The results are illustrated in Figure \ref{fig:2Dunstable}.

\begin{theorem}\label{th:Y2D_c1} Y-bounds - 2D\\
Assume truncation dimensions $M\in\mathbb{N}^2$ and $K>0$ and an approximate zero $\bar{p}$, with $\Pi_{MK}\bar{p} = \bar{p}$ to be given. Define  
\begin{equation}\label{eq:Y_MK}
Y_m^{MK} = 
|(Df^{MK}(\bar{p})\bar{p})_{m0}| + 2\displaystyle\sum_{k = 1}^{K}|(Df^{MK}(\bar{p})\bar{p})_{mk}|\nu^{k}, \quad m\preceq M, 
\end{equation}
and 
\begin{equation}\label{eq:Y_infty}
Y_{m}^{\infty} = \begin{cases}
\displaystyle2\sum_{k = K+1}^{2K}\frac{\alpha|((\bar{p}\ast_{TF}\bar{p})_m)_{k}|}{\tilde{\lambda}\cdot m+k^{2}-\alpha}\nu^{k} & |m|\geq 2, m\preceq M\\
\frac{\alpha |(\bar{p}\ast_{TF}\bar{p})_{m0}|}{\tilde{\lambda}\cdot m} + 2 \displaystyle\sum_{k = 1}^{3K}\frac{\alpha|((\bar{p}\ast_{TF}\bar{p})_m)_{k}|}{\tilde{\lambda}\cdot m+k^{2}-\alpha}\nu^k & M+1\preceq  m\preceq 2M,
\end{cases}.
\end{equation}
Then 
\begin{equation}\label{eq:YboundFisher1D} 
Y = \displaystyle\sum_{m_1 = 0}^{M_1} \displaystyle\sum_{m_2 = 0}^{M_2}Y^{MK}_m +\displaystyle\sum_{\stackrel{|m|\geq 2}{m\preceq M}}Y^{\infty}_m + \displaystyle\sum_{\stackrel{m_1>M_1\vee m_2>M_2}{m\preceq 2M}}Y^{\infty}_m
\end{equation}  fulfills \eqref{eq:Y}.
\end{theorem}

\begin{proof}
Analogue to Theorem \ref{th:Y1D_c1}.
\end{proof}

 \begin{figure}[h!] \label{fig:2Dunstable}
\begin{center}
\includegraphics[width=8cm]{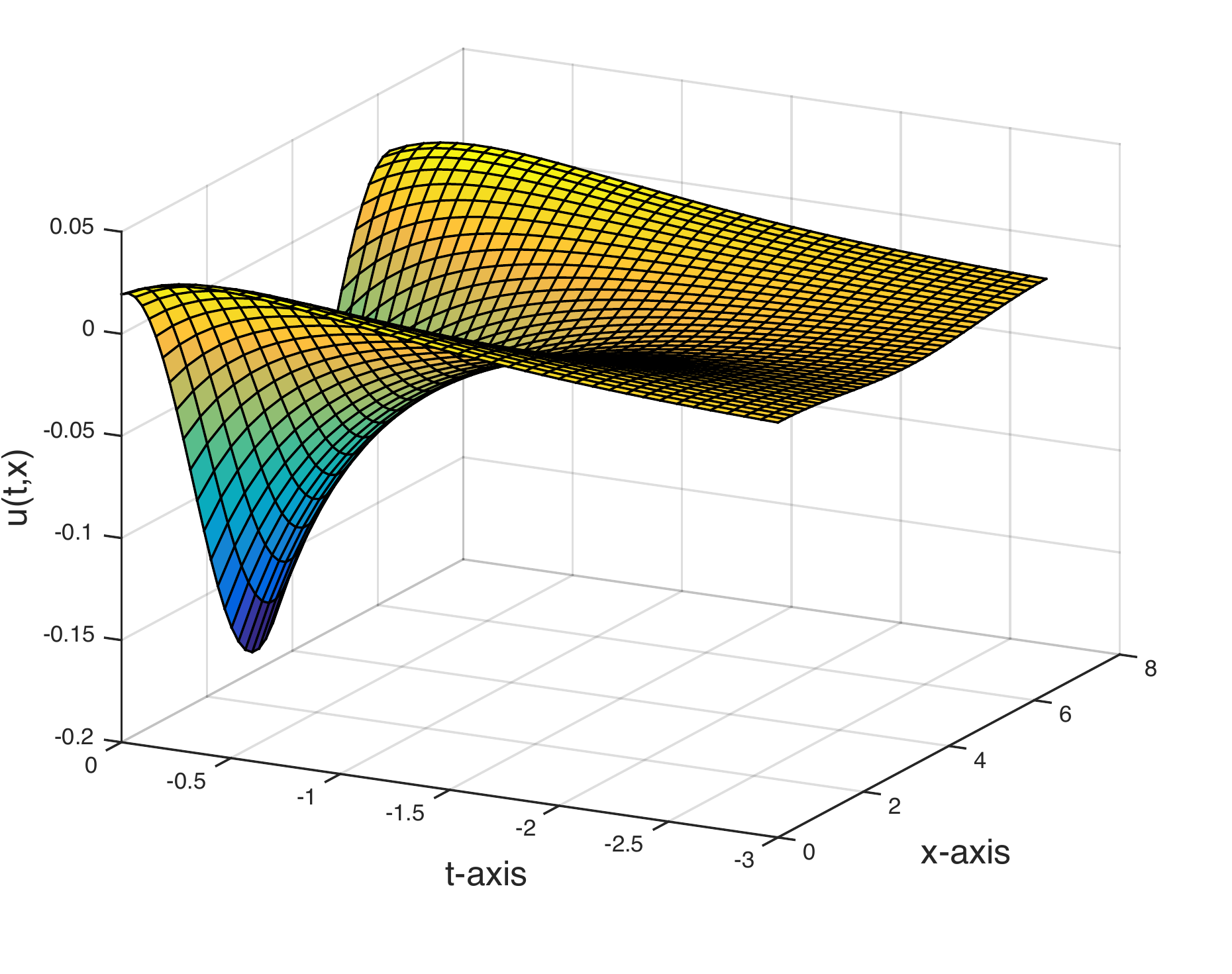}
\end{center}
\caption{We show the local unstable manifold computed by using that $a(t) = P\left(\exp(\tilde{\Lambda}t)\theta\right)$ is a solution for each $\theta\in\mathbb{B}_1$ $t\leq0$. Here $\tilde{\Lambda}\in\mathbb{R}^{2,2}$ the diagonal matrix with the unstable eigenvalues $\tilde{\lambda}_{1,2}$ as entries. We integrate P(0.01,0.95) backward for 3 time units and observe the convergence to the stationary solution $\tilde{u} = 0$. We scale the eigenvectors with $0.01$ and $0.05$ respectively.}  \label{fig:2Dunstable}
\end{figure}

\begin{theorem}\label{th:Z2D_c1} Z-bounds - 2D\\
Assume truncation dimensions $M\in\mathbb{N}^2$ and $K>0$ and an approximate zero $\bar{p}$, with $\Pi_{MK}\bar{p} = \bar{p}$ to be given. Define 
\begin{subequations}\label{eq:linear_data}
\label{eq:a0a1}
\begin{alignat}{1}
|\Delta^{(1)}|_{mk} &= \begin{cases}
0 & m = 0,e_i (i = 1,2)\quad k\geq 0\\
\displaystyle\sum_{l_1 = 0}^{m_1} \displaystyle\sum_{l_2 = 0}^{m_2}h_{lk}(\bar{p}) & |m|\geq 2, m\preceq M, 0\leq k\leq K\\
0 & (m\geq M+1,k\geq 0) \vee (0\leq m\leq M, k\geq K+1)
\end{cases}\\
|\Delta^{(2)}|_{mk}&=\begin{cases}
0 & m = 0,e_i (i = 1,2)\quad k\geq 0\\
\frac{2\alpha}{\nu^{k}} & 2\leq m\leq M, 0\leq k\leq K\\
0 & (m\geq M+1,k\geq 0) \vee (0\leq m\leq M, k\geq K+1)
\end{cases}\end{alignat}
\end{subequations},
\begin{equation}
\Sigma^{(j)}_{MK} = \displaystyle\sum_{m_1 = 0}^{M_1} \displaystyle\sum_{m_2 = 0}^{M_2}\left(|(|A||\Delta^{(j)}|)_0| + \displaystyle\sum_{k = 1}^{K}|(|A||\Delta^{(j)}|)_k|
\nu^k\right).
\end{equation}
and $\epsilon$ such that $\sup_{v\in B_{1}(0)}\|(Id-AA^{\dag})rv\|_{\nu}\leq \epsilon r$. Then 
\begin{equation}\label{eq:Z_boundFisher1D}
\begin{aligned}
Z(r) = &r\left(\epsilon + \Sigma_{MK}^{(1)}+ 2\alpha\|\bar{p}\|_{\nu}\left(\frac{1}{K^{2}-\alpha}+\frac{1}{\min(|\tilde{\lambda}_1|M_1,|\tilde{\lambda}_2|M_2)-\alpha}\right)\right)\\
&+r^{2}\left(\Sigma_{MK}^{(2)}+ 2\left(\frac{1}{K^{2}-\alpha}+\frac{1}{\min(|\tilde{\lambda}_1|M_1,|\tilde{\lambda}_2|M_2)-\alpha}\right)\right)
\end{aligned}
\end{equation}
fulfills \eqref{eq:Z}.
\end{theorem}
\begin{proof}
The proof is similar to the one of Theorem \ref{th:Z1D_c1}.
\end{proof}

\paragraph{Lower triangular structure}

To conclude this section let us remark on some implementation related issues. Recalling \eqref{eq:TF_convolution} to analyze the structure of $(X^{\nu,d},\ast_{TF})$ we see that $(p\ast_{TF}q)_m$ does only depend on terms with $\tilde{m}\preceq m$. This entails that the derivative $\frac{\partial}{\partial p_{\tilde{m}}}f_{m}\in\mathcal{L}(\ell^1_\nu,\ell^1_{\tilde{\nu}})$ $(1<\tilde{\nu}<\nu)$ is only non-zero for $\tilde{m}\preceq m$. We refer to this structure as lower triangular structure. We use this structure in our numerical implementation. To explain the idea let us restrict to the case $d = 1$, so there is a canoncial ordering on the indexing set $\mathbb{N}$ for $m$. Given truncation dimensions $M>0$ and $K>0$ the map $f^{MK}$ is a nonlinear map on $\mathbb{R}^{(M+1)(K+1)}$. Its derivative matrix $Df^{MK}(p)\in \mathbb{R}^{(M+1)(K+1),(M+1)(K+1)}$ has the following block structure:
\[
Df^{MK}(p) = \begin{pmatrix}
B_{00} & 0 & 0 \\
B_{10} & B_{11} & 0 & 0\\
\vdots & \vdots & \ddots & 0 \\
B_{M0} & B_{11} & \cdots & B_{MM}
\end{pmatrix},
\]
with $B_{ij} = \frac{\partial}{\partial p_{j}}f^{MK}_i(p)\in\mathbb{R}^{K+1,K+1}.$ Thus to solve a linear system of the form $Df^{MK}(p)h = b$, with $h,b\in\mathbb{R}^{(M+1)(K+1)}$ we can implement a block backsubstitution algorithm, that necessitates the solution of $M+1$ systems of size $(K+1)$ instead of the solution  of one system of size $(M+1)(K+1)$. So the complexity is $O((K+1)^3)$ against $O(((M+1)(K+1))^3)$ In the case that $d>1$ we see that this translates to the solution of $\prod_{i = 1}^d M_i$ systems of size $(K+1)$ making the advantage drawn from this way of implementation even more crucial.

\subsection{Computer assisted proof of a heteroclinic connecting orbit}\label{sec:connorbit}

In this section we demonstrate how our technique can be used to prove the existence of a connecting orbit. We restrict our attention to the constant case $c(x) = 1$.  In this case the fixed point $\frac{1}{c(x)}$ reduces in Fourier space to $\tilde{a}^1 = (1,0,0,\ldots)$. We see directly from the linearization of $g$ around $\tilde{a}^1 = (1,0,0,\ldots)$ that $\tilde{a}^1$ is spectrally stable, as all eigenvalues are negative. In addition we can derive nonlinear stability information in terms of an attracting neighborhood of $\tilde{a}^1$, see Lemma \ref{lem:attr_domain} below.  
\begin{lemma}\label{lem:attr_domain}
The  equilibrium point $\tilde{a}^1$ of  \eqref{eq:Fishers_seq} is an attracting fixed point with attracting neighborhood 
\begin{equation}
\mathcal{A} = \left\{a\in \ell^{1}_{\nu}: |a-\tilde{a}^1|_{\nu}<1\right\}.
\end{equation} 
\end{lemma}
\begin{proof}
See Appendix \ref{App1}.
\end{proof}
\begin{figure}[h!] 
\begin{center}
\includegraphics[width=10cm]{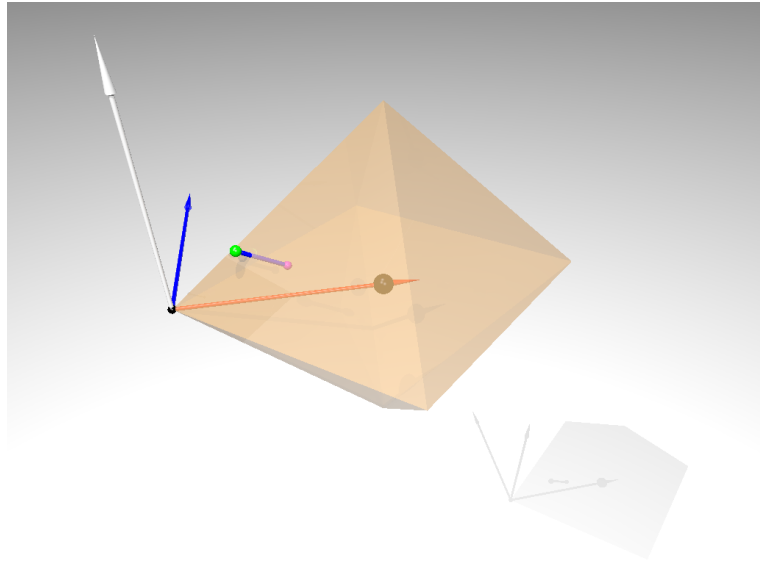}
\end{center}
\caption{
The green dot shows the projection of the non-trivial fixed point $\tilde{a}^2$ to the $(a_0,a_1,a_2)$ coordinate plane. The blue line is the part of the 1D local unstable manifold of $\tilde{a}^2$ that we validate whose endpoint is marked in purple and lies inside the domain of attraction of the sink  
} \label{fig:fisherconnection}
\end{figure}
\noindent This provides us with the ingredient to state the following Theorem proving the existence of a connecting orbit from the equilibrium $\tilde{a}^2$ with Morse index 1 to the sink $(1,0,0,\ldots)$. Note that we can use Lemma \ref{thm:fisherEquilibRadPolyBounds}, \ref{lem:stabilityCheck} and \ref{lem:eigCountFisher} as well as  Theorem \ref{th:Y1D_c1} and \ref{th:Z1D_c1} to compute this equilbirium $\tilde{a}^2$ together with its stability information. This is achieved in the Matlab script \verb|a_validateLinearDatac1_paperVersion.m| at the parameter value $\alpha = 2.1$. 
\begin{theorem}
Let $\nu = 1.1$. Let $P$ be a parametrization of the unstable manifold of $\tilde{a}^2$ with $|DP(0)|_\nu = 0.68194897863182\pm10^{-16}$. For $\theta = -0.505050505050505\pm10^{-16}$ we have that $P(\theta)\in\mathcal{A}$. Hence there is a connecting orbit from $\tilde{a}^2$ to $\tilde{a}^1$.
\end{theorem} 

\begin{proof}
By using the inequality 
\[
|P(\theta)-\tilde{a}^1|_{\nu}\leq  |P(\theta)-P^{MK}(\theta)|_{\nu} + |P^{MK}(\theta)-\tilde{a}^1|_{\nu}\leq r_{P} + |P^{MK}(\theta)-\tilde{a}^1|_{\nu}
\] 
we can rigorously check for any given $\theta$ if the true image $P(\theta)$ lies in $\mathcal{A}$. This computation is carried out using Matlab and Intlab in the program. In case of success it follows immediately from Lemma \ref{lem:attr_domain} that there is a connecting orbit from $\tilde{a}^2$ to $\tilde{a}^1$. The Matlab script \verb|proofconnection.m| available at \cite{parmPDEcode} carries out this check. It is called at the end of the script \verb|unstablec1_intval.m|, which computes and validates the parametrization.

\end{proof}

\section{Acknowledgments}
The authors would like to thank Mr. Jonathan Jaquette for suggesting to us the 
argument given in Appendix A.  Conversations with J.B. van den Berg, J.P. Lessard
and Rafael de la Llave were also extremely valuable.    
J.D.M.J. was partially supported by NSF grant DMS-1318172.

\appendix

\section{Domain of attraction of $\tilde{a} = (1,0,0,\ldots)$}\label{App1}
Lemma \ref{lem:attr_domain} in Section \ref{sec:connorbit} states that the  equilibrium point $\tilde{a} = (1,0,0,\ldots)$ of  \eqref{eq:Fishers_seq} is an attracting fixed point with attracting neighborhood 
\begin{equation}
\mathcal{A} = \left\{a\in \ell^{1}_{\nu}: |a-\tilde{a}|_{\nu}<1\right\}.
\end{equation} 
\begin{proof}
Write $a = \tilde{a} + h$. Then $a' = g(a)$ if 
\begin{equation}\label{eq:pert_eq}
h_{k}' = -(k^2+\alpha)h_k -\alpha (h\ast h)_k =  (Lh)_k + N(h)_k,\quad k\geq 0,  
\end{equation}
where $L = Dg(\tilde{a})$ and $N(h) = g(\tilde{a}+h)-g(\tilde{a})-Lh$. Denote $h(0) = h_0$ and let $h(t)$ solve \eqref{eq:pert_eq} with initial condition $h_0$. We show that if $|h_0|_\nu<1$, then $\displaystyle\lim_{t\to\infty}|h(t)|_\nu = 0$. Define $\left(\mathrm{e}^{-Lt}h_0\right)_k\bydef \mathrm{e}^{-(k^2+\alpha)t}h_{0k}$. Then 
\begin{equation}\label{eq:contractionsemigroup}
|\mathrm{e}^{-Lt}h_0|_\nu\leq \mathrm{e}^{-\alpha t}|h_{0}|_{\nu}.
\end{equation} Using the variation of constants formula we have
\begin{equation}\label{eq:varconst}
h(t) = \mathrm{e}^{-Lt}h_0 + \int_{0}^{t}\mathrm{e}^{-L(t-s)}(h\ast h)(s)ds.
\end{equation}
Using \eqref{eq:contractionsemigroup} in \eqref{eq:varconst} we obtain
\begin{equation}\label{eq:normbound}
|h(t)|_{\nu}\leq \mathrm{e}^{-\alpha t}|h_0|_\nu + \int_{0}^{t}\mathrm{e}^{-\alpha(t-s)}\alpha\underbrace{|h\ast h|_{\nu}}_{\leq |h|_{\nu}|h|_{\nu}}(s)ds.
\end{equation}
Assume $|h_0|_{\nu}<r<1$. By continuity of $|h(t)|$ there is a $t_{1}>0$ such that $r\leq \max_{s\in[0,t_1]}\leq \rho_1<1$. Then for $t\in[0,t_1]$:
\begin{equation}
\mathrm{e}^{\alpha t}|h(t)|_{\nu}\leq |h_0|_\nu + \int_{0}^{t}\alpha \rho_1 \mathrm{e}^{\alpha s}|h|_{\nu}(s)ds
\end{equation}
Using Gronwall's inequality for the function $\mathrm{e}^{\alpha t}|h(t)|_{\nu}$ we obtain 
\begin{equation}
\mathrm{e}^{\alpha t}|h(t)|_{\nu}\leq |h_0|_\nu \mathrm{e}^{\int_{0}^{t}\rho_1\alpha ds},
\end{equation}
hence $|h(t_1)|_\nu\leq |h_0|\mathrm{e}^{-\alpha(1-\rho_1)t_1}<|h_0|_{\nu}$. Inductively we construct a sequence of times $t_{k}$ with $\lim_{k\to\infty}t_k = \infty$ and $|h(t_k)|<|h(t_{k-1})|$ for $k\geq 2$. (By continuity of $|h(t)|_{\nu}$, hence if $t_k\to t^{\infty}<\infty$, then $|h(t)|_{\nu}$ would not be continuous in $t^{\infty}$.) 
As $(|h(t_{k})|_{\nu})_{k\in\mathbb{N}}$ is decreasing and bounded from below it converges to $0\leq \delta<1$. Assume $\lim_{k\to\infty}|h(t_k)| = \delta>0$.  There exists a $K>0$ such that $|h(t_k)|<\frac{1-\delta}{2}$ for all $k\geq K$. Then $|h(t)|_\nu\leq |h_0|\mathrm{e}^{-\alpha(1-\rho_K)t}<|h_0|_{\nu}$ for all $t\geq t_{K}$. For $t\to\infty$ this yields $\delta<0$, a contradiction. Hence $\delta = 0$.
\end{proof}

\bibliographystyle{plain}

\bibliography{references}

\end{document}